\documentclass[smallextended]{svjour3}       
 \pdfoutput=1 
\smartqed  
\usepackage{graphicx}
\usepackage{tikz}
\usepackage{amsmath,amsfonts,amssymb}
\usepackage{algorithm,algorithmic}
\newcommand{\C}{\mathbb{C}}    
\newcommand{\N}{\mathbb{N}}    
\newcommand{\R}{\mathbb{R}}    
\newcommand{\Z}{\mathbb{Z}}    

\newcommand{\la}{\left\langle}
\newcommand{\ra}{\right\rangle}

\newcommand{\setsep}{\;:\;}     

\newcommand{\cH}{\mathcal{H}}  
\newcommand{\cB}{\mathcal{B}}  
\newcommand{\cI}{\mathcal{I}}  
\newcommand{\cV}{\mathcal{V}}  
\newcommand{\cW}{\mathcal{W}}  

\newcommand{\diam}{\mathrm{diam}}
\newcommand{\supp}{\mathrm{supp}}
\newcommand{\diag}{\mathrm{diag}}
\newcommand{\spn}{\mathrm{span}}
\newcommand{\Id}{\mathrm{I}}

\newcommand{\gph}{\mathcal{G}} 
\newcommand{\dG}{\mathrm{deg}}  

\begin{document}

\title{Adaptive directional Haar tight framelets on bounded domains for digraph signal representations 
}

\titlerunning{Adaptive directional Haar tight framelets for digraph signal representations}        

\author{Yuchen Xiao         \and Xiaosheng Zhuang
}


\institute{Y. Xiao and X. Zhuang (Corresponding Author) \at
              Department of Mathematics, City University of Hong Kong, Tat Chee Avenue, Kowloon Tong, Hong Kong.
              \email{yuchexiao2@cityu.edu.hk; xzhuang7@cityu.edu.hk}           
              }

\date{Received: date / Accepted: date}

\maketitle

\begin{abstract}
Based on  hierarchical partitions,  we provide the  construction of Haar-type tight framelets on any  compact set $K\subseteq \R^d$. In particular, on the unit block $[0,1]^d$,  such tight framelets can be built to be with adaptivity and directionality.  We show that the adaptive directional Haar tight framelet systems can be used for  digraph signal representations. Some examples are provided to illustrate results in this paper.


\keywords{directional Haar tight framelets \and adaptive systems \and bounded domains\and graph signal processing \and digraph signal  \and graph clustering \and coarse-grained chain \and network \and deep learning \and machine learning}
\end{abstract}


\section{Introduction and motivation}
\label{sec:intro}

Harmonic analysis including Fourier analysis, frame theory, wavelet/framelet analysis, etc., has been one of the most active research areas in mathematics over the past two centuries \cite{Book:Stein1993}. Typical harmonic analysis is on theory and applications related to functions defined on regular Euclidean domains \cite{Book:Chui,CDV,Book:Daubechies1992,Book:Han,HM:AA,HanZhuang:alg.num,Book:Shearlets,Book:Mallat}. In recent years, driven by the rapid progress of deep learning and their successful applications in solving AI (artificial intelligence) related tasks, such as natural language processing, autonomous systems, robotics,  medical diagnostics, and so on, there has been a great interest in developing harmonic analysis  for data defined on non-Euclidean domains such as manifold data or graph data, e.g., see \cite{Bronstein_etal2017,qiyu2,treepap,CMZ:ACHA:18,Dong2017,qiyu3,hammond,Qiyu,Pesenson,Pesenson2,WaZh2018} and many references therein. For example,  data in machine/statistical learning, are typically from social networks, biology, physics, finance, etc., and can be naturally obtained or organized as graphs or graph data. Such data can be regarded as samples from an underlying manifold, where its \emph{graph Laplacian} is connected to the \emph{manifold Laplacian} encoding the essential information of the data to be exploited by various machine/deep learning approaches \cite{singer}. One also refers this area as the \emph{graph signal processing} (GSP) in contrast to signal/image/video processing \cite{Bronstein_etal2017}.

For graph signal analysis, the underlying graphs are typically directed graphs (or digraphs). For example, the citation networks modelling  relations among papers (as nodes) are digraphs where a paper can either be cited or cite other papers, which indicates a directed edge between two nodes; the information networks \cite{newman2003structure} with nodes consisting of URLs of web pages are digraphs, where an edge means that there is an URL in one web page linking to another web page; the traffic networks \cite{han2012extended} in modern cities with nodes representing intersections and edges representing traffic flows from one node to another are digraphs; the human body networks,  the nervous systems, and biological networks, etc., are all digraphs. The interested reader can refer to  \cite{malliaros2013clustering,przulj2011introduction,smith2013role} and many references therein. Similar to the wavelets and framelets for signal/image processing, multiscale representation systems based on various approaches such as spectral theory \cite{chung1997spectral}, diffusion wavelets\cite{achaspissue}, non-spectral construction \cite{treepap,CMZ:ACHA:18}, etc., have also been developed for graph signal representation and processing.

In this paper, motivated by the recent development of directional Haar framelet systems on $\R^d$ \cite{HLZ:AML,Li:DHF,Li:spie} as well as wavelet-like systems for graph signal processing \cite{Bronstein_etal2017,CMZ:ACHA:18,Wa:NN,WaZh:spie,Wa:ICML}, we focus on the development of directional multiscale representation systems for signals defined on digraphs. We are going to investigate the following two main problems:
\begin{enumerate}
\item[1)] How to construct directional Haar tight framelets on bounded domains with adaptivity?
\item[2)] How to efficiently represent digraph signals?
\end{enumerate}
In what follows, we lay out the main idea of this paper for the above two problems.  The details are given in the later sections.

For the first problem, we start with Haar wavelets.  Recall that for a separable Hilbert space $\mathcal{H}$, a collection $X=\{h_j\}_{j\in\N}\subseteq\mathcal{H}$ is said to be a \emph{frame} if there exist two positive constants $0<C_1\le C_2<\infty$ such that
\[
C_1\|f\|^2\le \sum_{j\in \N} \left|\la f, h_j\ra \right|^2\le C_2\|f\|^2\quad \forall f\in \mathcal{H},
\]
where $\la\cdot,\cdot\ra$ and  $\|\cdot\|$ are the inner product and norm in $\mathcal{H}$, respectively. If $C_1=C_2$, then such an $X$ is said to be \emph{tight}. If $C_1=C_2=1$ and $\|h_j\|=1$ for all $j$, then such an $X$ is an \emph{orthonormal basis} for $\mathcal{H}$.

Haar wavelet system \cite{Haar1910} is the first ever constructed orthonormal wavelet system on the interval $[0,1]$. It is a very simple yet elegant system that even  nowadays there are many literatures on Haar-type systems, e.g.,  see \cite{treepap,CMZ:ACHA:18,HLZ:AML,Wa:NN,Wa:ICML}. Starting from a scaling (refinable) function $\phi:=\chi_{[0,1]}$, which is a characteristic function defined on the unit interval $I:=[0,1]$, and the mother wavelet function $\psi:=\chi_{[0,\frac12)}-\chi_{[\frac12,1]}$, one can show that the system
\[
X(\R;\phi,\psi):=\{\phi(\cdot-k)\setsep k\in\Z\}\cup \{\psi_{j,k}:=2^{j/2}\psi(2^j\cdot-k)\setsep k\in\Z\}_{j\in\N_0}
\]
obtained from dilations and translations of $\phi$ and $\psi$, is an orthonormal wavelet basis for $L_2(\R)$ of square-integrable functions on $\R$ \cite{Book:Daubechies1992}. Here $\N_0:=\N\cup\{0\}$. Thanks to their compact support property, the restriction of such an orthonormal wavelet basis on the unit interval $I$ directly gives an orthonormal basis on the bounded domain for $L_2([0,1])$ of square-integrable functions on $I$:
\begin{equation}\label{D1Haar}
X(I;\phi,\psi):=\{\phi\}\cup\{\psi_{j,k}\setsep 0\le k < 2^j\}_{j\in\N_0}.
\end{equation}
This is indeed the system constructed by Haar in \cite{Haar1910}.

\begin{figure}[htbp!]
\centering
\includegraphics[width=3in]{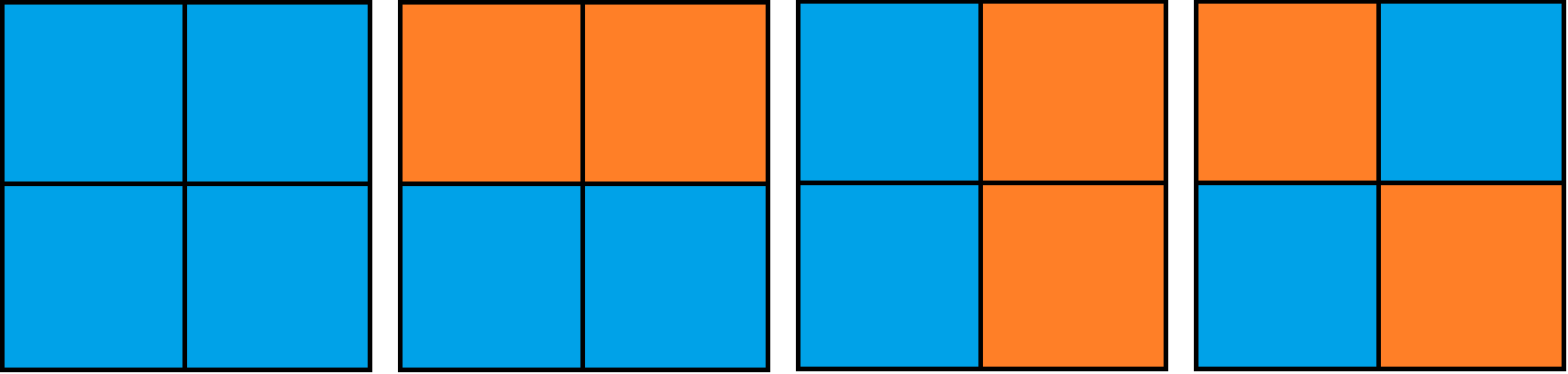}
\caption{The tensor product of 1D Haar wavelet scaling function  $\phi$ and wavelet function $\psi$. All are supported on the unit square $[0,1]^2$. The 4 big squares are (from left to right): $\phi\otimes \phi$, $\phi\otimes \psi$, $\psi\otimes \phi$, $\psi\otimes \psi$. Each colored sub-block represents either $1$ (blue) or $-1$ (orange) of the function value.}
\label{fig:D2Haar}
\end{figure}

In higher dimensions, the tensor product approach is usually employed to obtain orthonormal wavelets,  e.g., see Fig.~\ref{fig:D2Haar} for the $4$ generators $\phi\otimes \phi$, $\phi\otimes \psi$, $\psi\otimes \phi$, and $\psi\otimes \psi$ of the 2D orthonormal Haar wavelets. However, it is well-known that the tensor product orthonormal real-valued wavelets lack directionality \cite{CandesDonoho}, which hinders the sparsity representation of such systems for high-dimensional data analysis and their applications in image/video processing.
Various approaches including curvelets \cite{CandesDonoho}, shearlets \cite{AST2,shearlab,AST0,Book:Shearlets,AST1}, dual-tree complex wavelets \cite{DWT}, TP-$\C$TFs \cite{TPCTF2,TPCTF1,TPCTF3}, etc., on increasing directionality of  multiscale representation systems have been proposed over the past two decades, which we will not get into much of such developments but draw our attentions only to the main focus of this paper: \emph{Haar-type multiscale representation systems with directionality on  bounded domains}. Note that in order to increase directionality, one necessarily needs to consider  wavelet frames or framelet systems, which are more redundant representation systems than the orthonormal systems.

\begin{figure}[htbp!]
\centering
\includegraphics[width=4.5in]{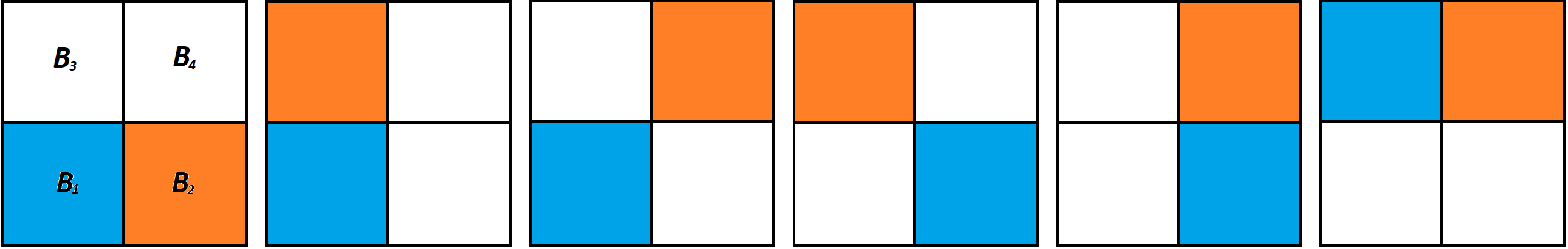}
\caption{The 6 functions in $\Psi$. Left to right: $\psi^{(1,2)},
\psi^{(1,3)},
\psi^{(1,4)}$,
$\psi^{(2,3)},
\psi^{(2,4)},
\psi^{(3,4)}
$. Each of them is supported on the unit square $[0,1]^2$, which is split to 4 sub-blocks $B_1,\ldots, B_4$.  Each colored sub-block represents either $1$ (blue) or $-1$ (orange) of the function value. White blocks mean 0 function value.}
\label{fig:D2DHF}
\end{figure}

In \cite{Li:DHF}, the authors proposed a new and simple  Haar-type directional system, the \emph{directional Haar tight framelets} (DHF) for $L_2(\R^2)$,  whose generators have 4 directions ($0^\circ, 90^\circ$, and  $\pm 45^\circ$). The system is generated from the scaling function $\varphi = \chi_I\otimes \chi_I$ and 6 generators in  $\Psi:=\{\psi^{(1,1)},\ldots, \psi^{(3,4)}\}$ defined by
\begin{equation}\label{DHF:generators}
\begin{aligned}
\psi^{(1,2)} &= \chi_{B_1}-\chi_{B_2},&
\psi^{(1,3)} &= \chi_{B_1}-\chi_{B_3},&
\psi^{(1,4)} &= \chi_{B_1}-\chi_{B_4},\\
\psi^{(2,3)} &= \chi_{B_2}-\chi_{B_3},&
\psi^{(2,4)} &= \chi_{B_2}-\chi_{B_4},&
\psi^{(3,4)} &= \chi_{B_3}-\chi_{B_4},&
\end{aligned}
\end{equation}
where $B_1:=[0,\frac12)\times [0,\frac12)$, $B_2:=[\frac12,1]\times[0,\frac12)$, $B_3:=[0,\frac12)\times [\frac12,1]$, and $B_4:=[\frac12,1]\times [\frac12,1]$ are the 4 sub-blocks obtained from  refining  the unit square $I^2 = [0,1]\times [0,1]=\cup_{\ell=1}^4 B_\ell$,  see Fig.~\ref{fig:D2DHF}. Clearly, compared to the 2D tensor product Haar wavelets (see Fig.~\ref{fig:D2Haar}), the DHF system has more directionality: the generators $\psi^{(1,2)}$ and $\psi^{(3,4)}$ can be used for vertical edge information extraction, the generators $\psi^{(1,3)}$ and $\psi^{(2,4)}$ can be used for horizontal edge information extraction, and the generators $\psi^{(1,4)}$ and $\psi^{(2,3)}$ can be used for $\pm45^\circ$ edge information extraction.  Note that the labelling  $(\ell_1,\ell_2), 1\le\ell_1<\ell_2\le 4$, is with respect to the number of choices of choosing two sub-blocks from the four sub-blocks in $[0,1]^2$  (${4 \choose 2}=6$ , see Theorem~\ref{thm:dirHaar} for more general results). The system defined by
\[
X(\R^2;\varphi,\Psi):=
\{\varphi(\cdot-k)\setsep k\in\Z^2\}\cup\{\psi_{j,k}: k\in\Z^2, \psi\in\Psi\}_{j\in\N_0},
\]
where $\psi_{j,k}:=2^{j/2}\psi(2^j\cdot-k)$,
is a tight frame  for $L_2(\R^2)$.  Its restriction to the unit square $I^2=[0,1]^2$ can be shown as
\begin{equation}\label{D2DHF}
X(I^2;\varphi, \Psi):=
\{\varphi\}\cup\{\psi_{j,k}: k=(k_1,k_2), 0\le k_1,k_2<2^j; \psi\in\Psi\}_{j\in\N_0}.
\end{equation}
This system $X(I^2; \varphi,\Psi)$
is indeed also  a tight frame for $L_2([0,1]^2)$ (see Theorem~\ref{thm:dirHaar}). Such a tightness property on $[0,1]^2$ is not explicitly shown in \cite{Li:DHF} nor in \cite{HLZ:AML}. In \cite{HLZ:AML}, the authors  further generalized such directional Haar tight framelets to arbitrary dimension $\R^d$.

By inspecting the structure of the system, we can  regroup $X(I^2;\varphi,\Psi)$ as
\begin{equation}\label{D2DHF:regroup}
X(I^2;\varphi,\Psi)=\{\varphi\}\cup \bigcup_{j=0}^\infty\bigcup_{k_1,k_2=0}^{2^j-1}\Psi_{j,(k_1,k_2)},
\end{equation}
where each
\[
\Psi_{j,k}:=\{2^{j/2}\psi^{(\ell_1,\ell_2)}(2^j\cdot-k)\setsep 1\le \ell_1<\ell_2\le4\}
\]
has 6 framelet functions supported on a sub-block
\[
B_{j,(k_1,k_2)}=[2^{-j}k_1,2^{-j}(k_1+1)]\times[2^{-j}k_2,2^{-j}(k_2+1)]\subseteq I^2
\]
at level $j$. Each $B_{j,(k_1,k_2)}$  is further refined to 4 sub-blocks  $B_{j+1,(2k_1,2k_2)}$, $B_{j+1,(2k_1+1,2k_2)}$, $B_{j+1,(2k_1,2k_2+1)}$, and $B_{j+1,(2k_1+1,2k_2+1)}$ at level $j+1$, see Fig.~\ref{fig:blocks} for the illustration. In other words, the system in \eqref{D2DHF:regroup} is based on a \emph{hierarchical partition} of the unit square $I^2$.  This  point of view together with how to sparsely represent digraph signals motivates us the main result in Theorem~\ref{thm:dirHaar}, where the question boils down to the construction of directional Haar tight framelet systems with \emph{adaptivity}.  Here, by ``adaptivity'' we mean that the blocks are not necessary square sub-blocks. We  affirmatively  show that  a system $X(\{\cB_j\}_{j\in\N_0})$,  associated with a sequence $\cB_j$ with each $\cB_j$ being a collection of subsets of a compact set $K\subseteq\R^d$ from a refining process, could be built to be a tight frame for $L_2(K)$. When $K=I^2=[0,1]^2$, such a system is our adaptive directional Haar tight framelets and it plays a key role in our second problem of efficient representations of digraph signals.

\begin{figure}[htbp!]
\centering
\includegraphics[width=0.9in]{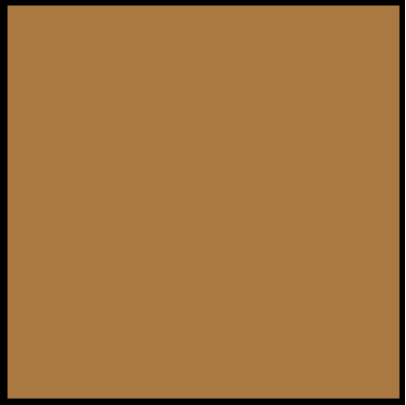}\quad
\includegraphics[width=0.9in]{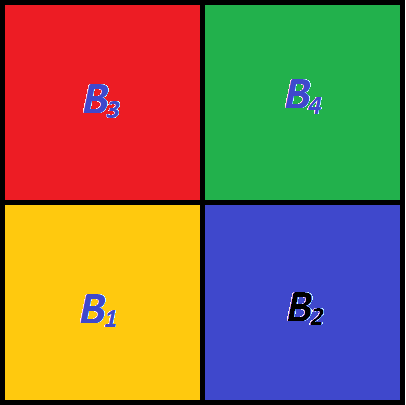}\quad
\includegraphics[width=0.9in]{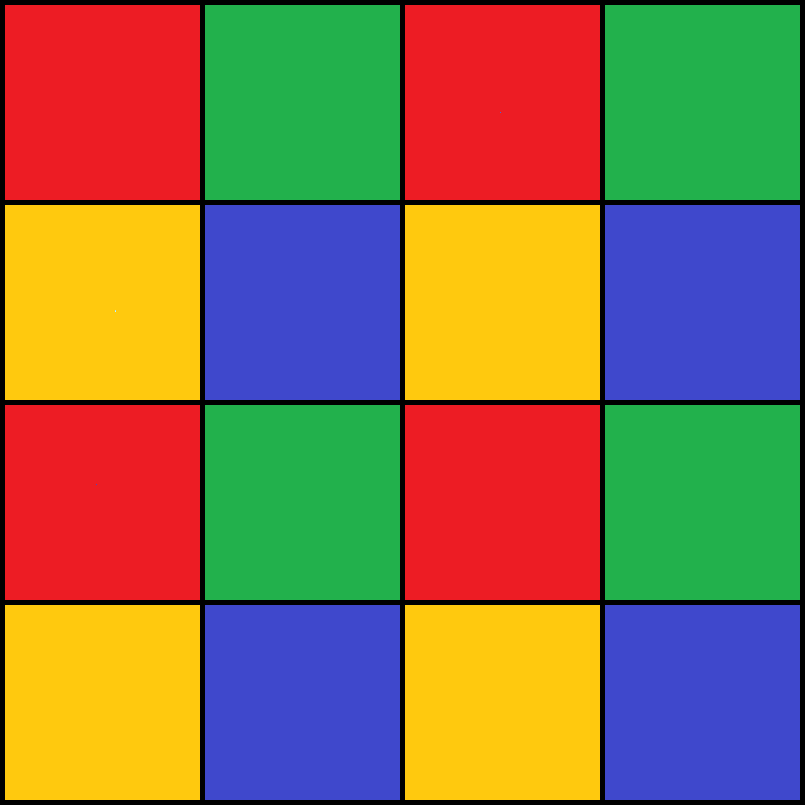}
\caption{The  unit square $I^2=[0,1]^2$ is refined to 4 sub-blocks $B_1,\ldots, B_4$. Each block $B_\ell$ is further refined to $4$ sub-blocks, and so on. Left: the unit square $I^2$ is associated with $\varphi$ and $\psi\in\Psi$ at level $j=0$. Middle:  $4$ refined blocks $B_1,\ldots, B_4$ are associated with $\Psi_{1,k}$ at level $j=1$. Right: 16 refined blocks are associated with $\Psi_{2,k}$ at level $j=2$.}
\label{fig:blocks}
\end{figure}

Now, continue to the second problem of efficient representations of signals on digraphs (directed graphs). We begin with undirected graphs. Recall that a \emph{graph} is an ordered pair $\gph=(V,W)$ with a nonempty set $V=\{v_1,\ldots,v_n\}$ of vertices and  an (weighted) \emph{adjacency matrix} $W: V\times V\rightarrow [0,\infty)$ of size $n\times n$ indicating edges  between vertices ($W(v_i,v_j)\neq 0$ if  there is an edge from the vertex $v_i$ to  $v_j$; otherwise 0). If the edges are unordered, that is, the edge from $v_i$ to $v_j$ is considered to be the same as the edge from $v_j$ to $v_i$, in which case, the matrix $W$ is symmetric,  then $\gph$ is said to be \emph{undirected}; otherwise, it is called a \emph{directed graph} or \emph{digraph},  see Fig.~\ref{fig:graphs} for an example of undirected graph and digraph. A \emph{signal} defined on a graph $\gph$ (or graph signal) is a function $f: V\rightarrow \C$.

\begin{figure}[htpb!]
\begin{minipage}{\textwidth}
\begin{center}
\begin{tikzpicture}[
roundnode/.style={circle, draw=green!60, fill=green!5, very thick, minimum size=7mm},
squarednode/.style={rectangle, draw=red!60, fill=red!5, very thick, minimum size=3mm},
]

\node[squarednode]      (v1)       at (-2.5,0)      {$a$};
\node[squarednode]      (v2)       at (-1.5,0)      {$b$};
\node[squarednode]      (v3)       at (0,0)         {$c$};
\node[squarednode]      (v4)       at (1,0)         {$d$};
\node[squarednode]      (v5)       at (2,0)         {$e$};
\node[squarednode]      (v6)       at (3.5,0)       {$f$};

\draw[->] (v1) -- (v2); \draw[]  (-2.0,0.2) node{1};
\draw[->] (v3) -- (v4);  \draw[]  (0.5,0.2) node{1};
\draw[->] (v4) -- (v5); \draw[]   (1.5,0.2) node{1};
\draw[<-] (v1.south).. controls (-2,-0.8) and (-0.5,-0.8) .. (v3.south);\draw (-1.1,-0.85) node{1};
\draw[<-] (v3.north).. controls (0.5,0.8) and (1.5,0.8)   .. (v5.north); \draw[] (1,0.85)     node{1};
\draw[->] (v3.south).. controls (0.8,-0.8) and (2.4,-0.8) .. (v6.south); \draw[] (1.9,-0.85)  node{1};
\draw[]  (5,0) node[circle,draw](g2){$\gph$};

\node[squarednode]      (v1)       at (-2.5,2)      {$a$};
\node[squarednode]      (v2)       at (-1.5,2)      {$b$};
\node[squarednode]      (v3)       at (0,2)         {$c$};
\node[squarednode]      (v4)       at (1,2)         {$d$};
\node[squarednode]      (v5)       at (2,2)         {$e$};
\node[squarednode]      (v6)       at (3.5,2)       {$f$};

\draw[-] (v1) -- (v2)  (-2.0,2.2) node{1};
\draw[-] (v3) -- (v4)   (0.5,2.2) node{1};
\draw[-] (v4) -- (v5)   (1.5,2.2) node{1};
\draw[-] (v1.south).. controls (-2,1.2) and (-0.5,1.2) .. (v3.south)  (-1.1,1.15) node{1};
\draw[-] (v3.north).. controls (0.5,2.8) and (1.5,2.8)   .. (v5.north)  (1,2.85)     node{1};
\draw[-] (v3.south).. controls (0.8,1.2) and (2.4,1.2) .. (v6.south)  (1.9,1.15)  node{1};
\draw[]  (5,2) node[circle,draw](g1){$\gph^x$};

\end{tikzpicture}
\end{center}
\end{minipage}\vspace{2mm}
\caption{An undirected graph $\gph^x=(V,W^x)$ (Top) and a digraph $\gph=(V,W)$ (Bottom) with the same vertex set $V=\{a,b,c,d,e,f\}$. Note that $W^x\neq W$ and $W^x$ is symmetric.}
\label{fig:graphs}
\end{figure}
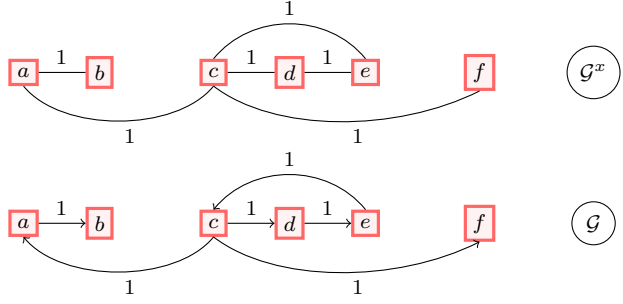

For a signal $f$ on an undirected graph $\gph=(V,W)$, one could identify it with a function on $I=[0,1]$ by  associating each vertex $v\in V$ a suitable subinterval $I_v\subseteq I$. In \cite{treepap}, the paper uses the concept of a \emph{filtration}, which a weight tree, for identifying vertices as subintervals in $I$ as well as building an (Haar-type) orthonormal basis on the filtration to represent signals on the underlying graph $\gph$. In this paper, we use the concept of the  \emph{coarse-grained chains} (\cite{Wa:NN,WaZh:spie,Wa:ICML}). Roughly speaking, a coarse-grained chain $\gph_{J\rightarrow0}:=(\gph_J,\gph_{J-1},\ldots,\gph_0)$ of $\gph\equiv\gph_J$ is a sequence of graphs such that $\gph_{j-1}$ is from the clustering result of  $\gph_{j}$. When $\gph_0$ has only one node, the coarse-grained chain is actually equivalent to a filtration in \cite{treepap}.  See Fig.~\ref{fig:coarse-grain-toy} for an example of a coarse-grained chain $\gph_{3\rightarrow0}:=(\gph_3,\ldots,\gph_0)$ of $\gph$. Each vertex in $\gph_{j-1}$ is a cluster of vertices in $\gph_{j}$. Based on such a coarse-grained chain, one can give a hierarchical representation $\{\cI_j\}_{j=0}^3$ of the interval $I$, where each $\cI_j=\{I_{j,k}\}$ is the collection of subintervals $I_{j,k}$ of $I$ such that $\cup_k I_{j,k} = I$, see Fig.~\ref{fig:coarse-grain-toy}.

Based on such a hierarchical sequence, one could build a Haar-type orthonormal basis \cite{treepap} for the function space $\spn\{\chi_{I_{3,k}}\setsep k = 1,\ldots,6\}$, which is the space for the signal defined on the graph. See Section~\ref{sec:example} for more details.

\begin{figure}[htpb!]
\begin{minipage}{\textwidth}
\centering
\begin{minipage}{\textwidth}
\begin{center}
\begin{tikzpicture}[
roundnode/.style={circle, draw=green!60, fill=green!5, very thick, minimum size=7mm},
squarednode/.style={rectangle, draw=red!60, fill=red!5, very thick, minimum size=3mm},
]

\node[squarednode]      (v1)       at (-2.3,0)      {$a$};
\node[squarednode]      (v2)       at (-1.5,0)      {$b$};
\node[squarednode]      (v3)       at (0.2,0)         {$c$};
\node[squarednode]      (v4)       at (1,0)         {$d$};
\node[squarednode]      (v5)       at (1.8,0)         {$e$};
\node[squarednode]      (v6)       at (3.1,0)       {$f$};

\draw[-] (v1) -- (v2)  (-1.9,0.2) node{1};
\draw[-] (v3) -- (v4)   (0.6,0.2) node{1};
\draw[-] (v4) -- (v5)   (1.4,0.2) node{1};
\draw[-] (v1.south).. controls (-2,-0.8) and (-0.1,-0.8) .. (v3.south)  (-0.8,-0.85) node{1};
\draw[-] (v3.north).. controls (0.5,0.8) and (1.5,0.8)   .. (v5.north)  (1,0.85)     node{1};
\draw[-] (v3.south).. controls (0.8,-0.8) and (2.4,-0.8) .. (v6.south)  (1.9,-0.85)  node{1};
\draw[]  (5,0) node[circle,draw](g3){$\gph_3$};
\draw[]  (-2.4,0.5) node{$[0,\frac16)$};
\draw[]  (-1.5,0.5) node{$[\frac16,\frac14)$};
\draw[]  (0,-0.4) node{$[\frac14,\frac{7}{12})$};
\draw[]  (1.0,-0.4) node{$[\frac{7}{12},\frac{9}{12})$};
\draw[]  (2.05,-0.4) node{$[\frac{9}{12},\frac{11}{12})$};
\draw[]  (3.1,0.5) node{$[\frac{11}{12},1]$};

\node[squarednode]      (v7)       at (-2,1.7)        {$a$\qquad $b$};
\node[squarednode]      (v8)       at (1,1.7)         {$c$\qquad $d$\qquad $e$};
\node[squarednode]      (v9)       at (3.1,1.7)       {$f$};

\draw[-] (v7) -- (v8)    (-0.75,1.9) node{1};
\draw[-] (v8) -- (v9)    (2.45,1.9)  node{1};
\draw[<->] (-2.5,2.0) arc (180:0:5mm);
\draw (-2,2.7) node{2};
\draw[<->] (0.2,2.0)  arc (150:30:9mm);
\draw (1,2.7)  node{6};
\draw[]  (5,1.7) node[circle,draw](g2){$\gph_2$};
\draw[]  (-2.0,1.3) node{$[0,\frac14)$};
\draw[]  (1.0,1.3) node{$[\frac14,\frac{11}{12})$};
\draw[]  (3.1,1.3) node{$[\frac{11}{12},1]$};

\node[squarednode]      (v10)       at (-2,3.4)        {$a$\qquad $b$};
\node[squarednode]      (v11)       at (1.75,3.4)         {$c$\qquad $d$ \qquad $e$\;\,\quad\qquad $f$};

\draw[-] (v10) -- (v11)    (-0.75,3.6) node{1};
\draw[<->] (-2.5,3.7) arc (180:0:5mm);
\draw (-2,4.4) node{2};
\draw[<->] (0.18,3.7)  arc (120:60:31mm);
\draw (1.6,4.4)  node{8};
\draw[]  (5,3.4) node[circle,draw](g1){$\gph_1$};
\draw[]  (-2.0,3.0) node{$[0,\frac14)$};
\draw[]  (1.75,3.0) node{$[\frac14,1]$};

\node[squarednode]      (v12)       at (0.35,5.1)         {$a$\qquad\;$b$\qquad\qquad\qquad$c$\qquad $d$\qquad $e$\quad\;\;\qquad $f$};
\draw[]  (5,5.1) node[circle,draw](g0){$\gph_0$};
\draw[]  (0,4.7) node{$[0,1]$};

\draw[->] (g3.north) -- (g2.south);
\draw[->] (g2.north) -- (g1.south);
\draw[->] (g1.north) -- (g0.south);

\end{tikzpicture}
\end{center}
\end{minipage}\vspace{2mm}
\begin{minipage}{0.8\textwidth}
\caption{A coarse-grained chain of $\gph$. $\gph_3$ is the underlying graph $\gph$. $\gph_{j-1}$ is from clustering of $\gph_j$ for $j=1,2,3$. Note that $\gph_0$ has one vertex only. Here each box represents a node (or cluster) in the graph, the  lines  represent edges between vertices, and the arc on a same node indicates a self-loop. $\gph_0$ can be identified as the root interval $I=[0,1]$, $\gph_1$ as $[0,\frac14)\cup [\frac14,1]$, $\gph_2$ as $[0,\frac14)\cup[\frac14,\frac{11}{12})\cup[\frac{11}{12},1]$, and $\gph_3$ as $[0,\frac{1}{6})\cup[\frac16,\frac{1}{4})\cup[\frac{1}{4},\frac{7}{12})\cup[\frac{7}{12},\frac{9}{12})\cup[\frac{9}{12},\frac{11}{12})\cup[\frac{11}{12},1]$.}
\label{fig:coarse-grain-toy}
\end{minipage}
\end{minipage}
\end{figure}

Returning to digraph signal representations, can one use similar approaches for undirected graph to represent digraph signals? The answer is \emph{yes and no}. For ``no'' it is because most of the clustering algorithms are developed based on the symmetry property of the adjacency matrix $W$ or the well-defined and well-understood operator on the undirected graph: the \emph{graph Laplacian} \cite{chung1997spectral}. It is not trivial to \emph{directly} use them for digraph cases. For ``yes'' it is because there are \emph{undirect} ways to circumvent such difficulties. In fact, one typical approach is to define a counterpart graph Laplacian on digraph, such as the \emph{Hodge Laplacian} in \cite{hodge_laplacian_lim2015}, the weighted adjacency matrix in \cite{chung_directed_laplacian}, the so-called \emph{dilaplacian} in \cite{li2012digraph}, and so on. In this paper, we use the idea developed in \cite{CMZ:ACHA:18}: to lift the dimension from one to two by using a pair of undirected graphs to represent a given digraph. In a nutshell, a digraph signal is identified as a signal defined on $I^2=[0,1]\times [0,1]$ through the following steps.

\begin{itemize}
\item[1)] In view of the singular value decomposition,  the adjacency matrix $W$ in a digrah  $\gph=(V,W)$ is uniquely determined by $WW^\top$ and $W^\top W$ from which one could construct a pair of undirected graphs $\gph^x = (V, WW^\top)$ and $\gph^y=(V,W^\top W)$.
\item[2)]  Applying well-known techniques, e.g., \cite{treepap}, for undirected graphs, one can represent vertices in each graph of $\gph^x$ and $\gph^y$ as subintervals in $I=[0,1]$.
\item[3)] Suppose a vertex $v$ is identified as a subinterval $I_v^x=[a,b)$ on $\gph^x$ and $I_v^y=[c,d)$ on $\gph^y$, then $v$ in the original digraph $\gph$ is identified as a block $[a,b)\times [c,d)\subseteq I^2$. Consequently, the vertices in the digraph are sub-blocks in the unit square.
\item[4)] Then, signals on $\gph$ can be viewed as functions defined on the unit square $[0,1]^2$.
\end{itemize}

In \cite{CMZ:ACHA:18}, once orthonormal bases are built for $\gph^x$ and $\gph^y$, then the tensor product approach is  used to construct orthonormal bases for $\gph$. As we pointed out, the tensor product approach lacks directionality. Since directional systems provide better sparse representations than those by the tensor product ones, in this paper, we  use our  adaptive directional Haar tight framelets in 2D for the digraph signal representations based on the above digraph representations $\gph\leftrightarrow(\gph^x,\gph^y)$.

The contribution of the paper is threefold. First, based on a hierarchical partition,  we provide a simply yet flexible construction of Haar-type tight framelets on any compact set $K\subseteq \R^d$. Second, such framelet systems include directional Haar systems in \cite{HLZ:AML,Li:DHF} as special cases and lead to the adaptive directional Haar tight framelet systems for non-dyadic partitions of the unit block $[0,1]^d$. Last but not least,  we demonstrate that digraph signals can be identified as signals defined on the unit square $[0,1]^2$ and hence could be efficiently represented by the adaptive directional Haar tight framelet systems where the directionality is a really desired property.

The structure of the paper is as follows. In Section~\ref{sec:dirHaar}, we present our main results on the construction of tight frames for $L_2(K)$ for some compact set $K\subseteq \R^d$. Then, adaptive directional Haar tight framelets on bounded domains are deduced. In Section~\ref{sec:digraph}, we show how to represent digraph signals using the developed adaptive directional Haar tight framelets. In Section~\ref{sec:example}, we provide some examples to illustrate our main results. Conclusion and further remarks are given in Section~\ref{sec:remarks}. Some proofs are postponed to the last section.


\section{Adaptive directional Haar tight framelets on bounded domains}
\label{sec:dirHaar}
Let $K\subseteq \R^d$ be a compact set  and consider the Hilbert space $L_2(K):=\{f\setsep \|f\|_2:=(\int_K |f(x)|^2dx)^{\frac12}<\infty\}$ of square-integrable functions $f$ on $K$. The  inner product on $L_2(K)$ is defined by $\la f,g\ra:=\int_K f(x)\overline{g(x)}dx$ for $f,g\in L_2(K)$.  In this section, based on a hierarchical partition of $K$,
 we construct a system $X=\{\varphi\}\cup\{\Psi_j\}_{j\in\N_0}$of elements in $L_2(K)$ and show that it is a tight frame for $L_2(K)$. Such a system $X$ leads to our \emph{adaptive directional Haar tight framelets} (AdaDHF) on $K$.

Before we present and prove our main result in Theorem~\ref{thm:dirHaar}, let us introduce some necessary notation, definitions, and auxiliary results first.  For a Hilbert space $\cH$, the collection $X=\{h_j\}_{j\in\N}\subseteq \cH$  is a \emph{tight frame} for $\cH$ if
\begin{equation}\label{def:tight1}
\|f\|^2 = \sum_{j\in\N}\left|\la f, h_j\ra\right|^2 \quad\forall f\in\cH.
\end{equation}
Using the polarization identity, one can show that it is equivalent to
\begin{equation}\label{def:tight2}
f = \sum_{j\in\N} \la f, h_j\ra h_j\quad \forall f\in\cH.
\end{equation}

We denote $\Id_m$  the identity matrix of size $m\times m$. The matrix $A$ in the following  lemma is used  to connect functions on two scales supported on a same block $B\subseteq K$. Its proof is  postponed to Section~\ref{sec:proofs}.
\begin{lemma}
\label{lem:A}
Let $m\in\N$ and  $b_1$, $\ldots$, $b_m$ be $m$ positive constants such that $\sum_{\ell=1}^m b_\ell =1$.  Let $n={m\choose 2}$ and
$A=(a_{i,\ell})_{0\le i\le n; 1\le \ell \le m}$ be a matrix of size $(n+1) \times m$ of the form:
\begin{equation}\label{matrix:A}
A =
\left[
\begin{matrix}
\sqrt{b_1} & \sqrt{b_2} & \sqrt{b_3} & \cdots & \sqrt{b_{m-1}} & \sqrt{b_{m}}\\
\sqrt{b_2} & -\sqrt{b_1} & 0 & \cdots & 0 & 0\\
\sqrt{b_3} & 0 & -\sqrt{b_1} & \cdots & 0 & 0\\
\vdots & \vdots & \vdots & \ddots & \vdots & \vdots\\
0 & 0 & 0 & \cdots & \sqrt{b_{m}} & -\sqrt{b_{m-1}}\\
\end{matrix}
\right].
\end{equation}
That is, the first row of $A$ (with respect to $i=0$) is
\[
(a_{0,\ell})_{\ell=1}^{m} = \left(\sqrt{b_\ell}\right)_{\ell=1}^{m},
\]
and the  row $(a_{i,\ell})_{\ell=1}^{m}$ for $i\neq 0$ is given by
\[
a_{i,\ell} = a_{{(i_1,i_2)},\ell}=
\begin{cases}
\sqrt{b_{i_2}} & \mbox{if }\ell = i_1,\\
-\sqrt{b_{i_1}} & \mbox{if } \ell = i_2,\\
0 & \mbox{otherwise,}
\end{cases}
\]
where for each $i\neq 0$, the index $i$ is uniquely determined by a pair $(i_1,i_2)$ satisfying $1\le i_1< i_2\le m$ through $(i_1,i_2)\mapsto i=\frac{(2m-i_1)(i_1-1)}{2}+i_2-i_1$.
Then, $A$ satisfies $A^\top A=\Id_m$.
\end{lemma}

\medskip

For  a measurable set $B\subseteq \R^d$, we denote  $|B|$ as its  Lebesgue measure and  $\chi_B$ as the characteristic function on $B$. The following lemma shows that we can construct a set of functions supported on $B$ so that it is tight.

\begin{lemma}
\label{lem:Psi_B}
Let $B\subseteq K\subseteq\R^d$ be a measurable subset in the compact set $K$ satisfying $|B|>0$ and   $B_\ell, \ell=1,\ldots, m$ with $m\ge2$ be measurable sub-blocks of $B$ such that
$B =\cup_{\ell=1}^mB_\ell$,  $|B_\ell|>0$ for all $\ell=1,\ldots, m$, and  $|B_{\ell_1}\cap B_{\ell_2}|=0$ for $\ell_1\neq \ell_2$.      Define the set $\Psi_B:=\{\psi^{(\ell_1,\ell_2)}\setsep 1\le \ell_1<\ell_2\le m\}$ of functions  by
 \begin{equation}\label{def:psi_ell}
      \psi^{(\ell_1,\ell_2)}:=\sqrt{b_{\ell_2}}\gamma_{\ell_1}-
      \sqrt{b_{\ell_1}}\gamma_{\ell_2}, \quad 1\le \ell_1<\ell_2 \le m,
 \end{equation}
 where $\gamma_{\ell}:=\frac{\chi_{B_{\ell}}}{\sqrt{|B_{\ell}|}}$  and $b_\ell:=\frac{|B_\ell|}{|B|}$. Then $\Psi_B$ is a tight frame for
 \[
 \cW_B:=\spn\{\psi^{(\ell_1,\ell_2)}\setsep 1\le \ell_1<\ell_2\le m \}.
 \] That is,
 \[
 f = \sum_{1\le \ell_1<\ell_2\le m}\la f, \psi^{(\ell_1,\ell_2)}\ra \psi^{(\ell_1,\ell_2)} \quad \forall f\in \cW_B.
 \]
\end{lemma}

\medskip

The proof of Lemma~\ref{lem:Psi_B} uses results in Lemma~\ref{lem:A} and it is one of the key steps in our proof of the main theorem. We postpone it to Section~\ref{sec:proofs}.

Note that $\psi^{(\ell_1,\ell_2)}$ in Lemma~\ref{lem:Psi_B} are constructed from $\chi_{B_\ell}$, $\ell=1,\ldots,m$. The following lemma demonstrates that those $\chi_{B_\ell}$'s can be constructed  from $\psi^{(\ell_1,\ell_2)}$'s  together with $\chi_B$ as well.

\begin{lemma}
\label{lem:VW}
Let $B$, $B_\ell, \gamma_\ell, \ell=1,\ldots, m$, and $\Psi_B:=\{\psi^{(\ell_1,\ell_2)}, 1\le \ell_1<\ell_2\le m\}$ be defined as  in Lemma~\ref{lem:Psi_B}. Define vectors $\Gamma_B$, $\Phi_B$ of functions by
\[
\Gamma_{B}:=(\gamma_\ell)_{\ell=1}^{m}\mbox{ and }\Phi_{B}:=(\gamma_B,\Psi_{B})=\left(\gamma_{B},\psi_1,\ldots,\psi_n\right),
\]
where $\gamma_B:=\frac{\chi_B}{\sqrt{|B|}}$ and $\psi_1,\ldots, \psi_n$ are from enumerating the elements in $\Psi_{B}$   through $(\ell_1,\ell_2)\mapsto \frac{(2m-\ell_1)(\ell_1-1)}{2}+(\ell_2-\ell_1)$ with $n = {m \choose 2}=\frac{m\times (m-1)}{2}$. Then
\[
\Gamma_B = A^\top \Phi_B.
\]
Consequently, the space $\cV_1:=\spn\{\gamma_\ell\setsep \ell=1,\ldots,m\}=\cV_B\oplus \cW_B$ where $\cV_B:=\spn\{\chi_B\}$ and $\cW_B:=\spn\{\psi\setsep \psi\in\Psi_B\}$.
\end{lemma}
\begin{proof}
Note that $\Gamma_{B}$ is a vector of size $m$ while $\Phi_{B}$ is a vector of size $n={m\choose 2}+1$. From the definition of $\psi^{(\ell_1,\ell_2)}$ in \eqref{def:psi_ell} and $B=\cup_{\ell}B_\ell$, it is easy to verify that
\[
\Phi_{B} = A\Gamma_{B},
\]
where $A=(a_{i,\ell})_{0\le i\le n; 1\le \ell \le m}$ is a matrix of size $(n+1) \times m$  defined as in Lemma~\ref{lem:A}.
By Lemma~\ref{lem:A}, we have $A^\top A=\Id_{m}$, which implies that $\Gamma_{B} = A^\top \Phi_{B}$. Hence, $\cV_1\subseteq \cV_B+\cW_B$.  Now the fact that $\cV_1=\cV_B\oplus \cW_B$ follows directly from $\cV_B\subseteq \cV_1, \cW_B\subseteq \cV_1$ and $\cV_B\perp \cW_B$. This completes the proof.  $\blacksquare$
\end{proof}

\medskip

By splitting the compact set $K$, one can obtain subsets $B_\ell$ of $K$. For each subset $B_\ell$, one can further refine it to have smaller subsets. Such a process could continue and one could obtain a hierarchical partition of $K$.  We say that the sequence $\{\cB_j\}_{j\in\N_0}$ is a \emph{hierarchical partition} of $K$ if it satisfies the following conditions:
\begin{itemize}
  \item[{\rm a)}] \emph{Root property}: each $\cB_j$ is a collection of finite number of measurable subsets of $K$ with $\cB_0=\{K\}$, $\cup_{B\in\cB_j} B = K$,  $|B|>0$ for all $B\in \cB_j$, and $|B_1\cap B_2|=0$ for any $B_1\neq B_2$  in $\cB_j$.

  \item[{\rm b)}] \emph{Nested property}: $\{\cB_j\}_{j\in\N_0}$ is \emph{nested} in the  sense that for each $B\in \cB_{j-1}$,  $B = \cup_{\ell=1}^{c_B} B_\ell$ with $B_\ell\in \cB_j$. That is $B_\ell$'s are children of $B$ in $\cB_j$ and the positive integer $c_B\ge1$ denotes the number of children of $B$ in $\cB_{j}$. In other words,   the sets in $\cB_j$ are obtained from the splitting of sets in $\cB_{j-1}$.

  \item[{\rm c)}] \emph{Density  property}: $\lim_{j\rightarrow\infty}\diam(\cB_j)=0$ where  $\diam(\cB_j):=\max\{\diam(B)\setsep B\in\cB_j\}$ and $\diam(B):=\sup\{|x-y|\setsep x,y\in B\}$ is the diameter of the set $B$.
\end{itemize}

We are now ready to introduce and prove our main result.
\begin{theorem}\label{thm:dirHaar}
Let $K\subseteq \R^d$ be a compact set in $\R^d$ with $|K|>0$ and $\{\cB_j\}_{j\in\N_0}$ be a hierarchical partition of $K$.
Define the set
\[
X(\{\cB_j\}_{j\in\N_0}):=\{\varphi_0\}\cup \{\Psi_{j,B}\setsep B\in\cB_j\}_{j\in\N_0}
\]
of functions by $\varphi_0 := \frac{\chi_{K}}{\sqrt{|K|}}$
and  $\Psi_{j,B}:=\{\psi_{j,B}^{(\ell_1,\ell_2)}:  1\le \ell_1<\ell_2\le c_B\}_{j=0}^\infty$ with
      \begin{equation}\label{def:psijb}
      \psi_{j,B}^{(\ell_1,\ell_2)}:=\sqrt{b_{\ell_2}}\gamma_{\ell_1}-
      \sqrt{b_{\ell_1}}\gamma_{\ell_2}, \quad 1\le \ell_1<\ell_2 \le c_B,
      \end{equation}
 where  $B_\ell\in\cB_{j+1}$, $\ell=1,\ldots, c_B$ are the children sub-blocks of $B$,
 $\gamma_{\ell}:=\frac{\chi_{B_{\ell}}}{\sqrt{|B_{\ell}|}}$,  and $b_\ell:=\frac{|B_\ell|}{|B|}$.
Then, $X(\{\cB_j\}_{j\in\N_0})$ is a tight frame for $L_2(K)$.
\end{theorem}

\begin{proof}
By \eqref{def:tight1}, we need to prove that
\[
\begin{aligned}
\|f\|_2^2 &
=
|\la f, \varphi_0\ra|^2+\sum_{j=0}^\infty\sum_{B\in\cB_j}\sum_{1\le\ell_1<\ell_2\le c_B}\left |\la f, \psi^{(\ell_1,\ell_2)}_{j,B}\ra\right|^2\quad \forall f\in L_2(K).
\end{aligned}
\]
We proceed  through the following steps.

\begin{enumerate}
\item[1)] First, let $\cV_0:=\spn\{\chi_K\}=\spn\{\varphi_0\}$ and
\begin{equation}\label{def:Vj}
\cV_j:=\spn\{\chi_{B} \setsep B\in \cB_{j}\}
\end{equation}
for $j\in\N_0$. Then by the nested property of $\{\cB_j\}_{j\in\N_0}$, we have
\[
\cV_0\subseteq \cV_1\subseteq \cdots \subseteq \cV_j\subseteq \cV_{j+1}\subseteq \cdots.
\]
By the density property of  $\{\cB_j\}_{j\in\N_0}$, we see that $\cup_{j\in\N_0}\cV_j$ is dense in $L_2(K)$.

\item[2)] Let
\begin{equation}\label{def:Wj}
\cW_j:=\spn\{\psi: \psi\in \Psi_{j,B}, B\in\cB_j\}
\end{equation}
and
\[
\cW_{j,B}:=\spn\{\psi \setsep \psi\in\Psi_{j,B}\}, \quad j\in\N_0.
\]
Thanks to the nested property of $\{\cB_j\}_{j\in\N_0}$ and our construction of $\psi_{j,B}^{(\ell_1,\ell_2)}$, we have that for any $B\in \cB_j$
\[
\la \chi_B, \psi^{(\ell_1,\ell_2)}_{j,B}\ra = 0, \quad 1\le \ell_1<\ell_2\le c_B.
\]
Hence, we see that $\cV_j\perp \cW_j$ and $\cW_j\perp \cW_{j'}$ for all $j,j'\in\N_0$ and $j\neq j'$. Moreover, we claim that
\[
\cV_{j+1} = \cV_j \oplus \cW_j\quad \forall j\in\N_0.
\]
Obviously,  $\cV_j\subseteq \cV_{j+1}$ and $\cW_j\subseteq \cV_{j+1}$. Hence, we only need to show that $\cV_{j+1}\subseteq (\cV_j +\cW_j)$, which by the nested property and noticing $\cW_j=\oplus_{B\in\cB_j}\cW_{j,B}$ for all $j\in\N_0$, it suffices to show that for each $B\in \cB_j$, functions in
\[
\{\chi_{B_\ell}\setsep B_\ell\in\cB_{j+1} \mbox{ are children of } B  \}\subseteq \cV_{j+1}
\]
are the linear combinations of functions in
\[
\{\chi_B\}\cup\{\psi_{j,B}^{(\ell_1,\ell_2)} \setsep 1\le \ell_1<\ell_2\le c_B\} \subseteq (\cV_j+\cW_j),
\]
which follows from Lemma~\ref{lem:VW}.
Therefore $\cV_{j+1} = \cV_j\oplus \cW_j$ for all $j\in\N_0$.

\item[3)] Consequently, $\cV_0\oplus\bigoplus_{j\in\N_0,B\in\cB_j}\cW_{j,B}$ is dense in $L_2(K)$. Hence,  for each $f\in L_2(K)$, there exists a sequence  $\{c_{\varphi_0}\}\cup\{c_{j,B}^{(\ell_1,\ell_2)}\setsep B\in\cB_j, 1\le\ell_1,\ell_2\le c_B \}_{j\in\N_0}$ of constants such that
\[
f=c_{\varphi_0} \varphi_0 +\sum_{j=0}^\infty\sum_{B\in \cB_j} \sum_{1\le\ell_1<\ell_2\le c_B}c_{j,B}^{(\ell_1,\ell_2)} \psi^{(\ell_1,\ell_2)}_{j,B},
\]
where the equality holds in the $L_2$-sense. Define
\[
f_{j,B}:= \sum_{1\le\ell_1<\ell_2\le c_B}c_{j,B}^{(\ell_1,\ell_2)} \psi^{(\ell_1,\ell_2)}_{j,B}\in \cW_{j,B}.
\]
Then $f=c_{\varphi_0}\varphi_0+\sum_{j=0}^\infty\sum_{B\in\cB^j} f_{j,B}$ and we have
\[
\begin{aligned}
\|f\|_2^2 &= \la f,f\ra =
|c_{\varphi_0}|^2 +\sum_{j=0}^\infty\sum_{B\in \cB_j} \|f_{j,B}\|_2^2,
\end{aligned}
\]
where the series converges absolutely. On the other hand, we have
\[
\begin{aligned}
\sum_{h\in X(\{\cB_j\}_{j\in\N_0}} |\la f, h\ra|^2
&=|\la f,\varphi_0\ra|^2+\sum_{j=0}^\infty\sum_{B\in \cB_j} \sum_{1\le\ell_1<\ell_2\le c_B} \left|\la f, \psi^{(\ell_1,\ell_2)}_{j,B}\ra\right|^2\\
&=|\la f,\varphi_0\ra|^2+\sum_{j=0}^\infty\sum_{B\in \cB_j} \sum_{1\le\ell_1<\ell_2\le c_B} \left|\la f_{j,B}, \psi^{(\ell_1,\ell_2)}_{j,B}\ra\right|^2.
\end{aligned}
\]
Hence, to prove that $\|f\|_2^2 = \sum_{h\in X(\{\cB_j\}_{j\in\N_0})}\left|\la f, h\ra\right|^2$, it suffices to show that for each $j\in\N_0$ and $B\in \cB_j$, we have
\[
\|f_{j,B}\|_2^2=\sum_{1\le\ell_1<\ell_2\le c_B} \left|\la f_{j,B}, \psi^{(\ell_1,\ell_2)}_{j,B}\ra\right|^2.
\]
This is equivalent to showing that $\Psi_{j,B}=\{\psi_{j,B}^{(\ell_1,\ell_2)}\setsep 1\le \ell_1,\ell_2\le c_B\}$ is a tight frame for $\cW_{j,B}$, which  follows from Lemma~\ref{lem:Psi_B}.
\end{enumerate}

Consequently, we prove that $X(\{\cB_j\}_{j\in\N_0})$ is a tight frame for $L_2(K)$.  $\blacksquare$
\end{proof}

\medskip

The system $X(\{\cB_j\}_{j\in\N_0})$ in Theorem~\ref{thm:dirHaar}, which depends only on the hierarchical partition of $K$,  is very flexible. Next, we discuss some of its special cases.

In practice, signals usually lie in finite-dimensional spaces. Hence, it is useful to study  the cut-off system $X(\{\cB_j\}_{j=0}^J)$ up to some scale $J\in\N_0$.

\begin{corollary}
\label{cor:cut-off} Retaining all assumptions and notation in Theorem~\ref{thm:dirHaar}. Given $J\in\N_0$, define the cut-off system $X(\{\cB_j\}_{j=0}^J)$ by
\begin{equation}
\label{def:cut-off}
X(\{\cB_j\}_{j=0}^J):=\{\varphi_0\}\cup\{\Psi_{j,B}\setsep B\in\cB_j\}_{j=0}^{J-1}.
\end{equation}
Then, $X(\{\cB_j\}_{j=0}^J)$ is a tight frame for $\cV_{J}$ defined as in \eqref{def:Vj}.
\end{corollary}

\begin{proof}
Note that
$\cV_{J}=\cV_0\oplus\bigoplus_{j=0}^ {J-1} \cW_{j}$. The conclusion follows similarly to the proof of Theorem~\ref{thm:dirHaar} by showing that $\|f\|_2^2=\sum_{h\in X(\{\cB_j\}_{j=0}^J)}|\la f, h\ra|^2$ for all $f\in \cV_J$.
\end{proof}

We immediately have the following corollary if each splitting of a block $B$ has at most two children (sub-blocks).

\begin{corollary}
\label{cor:ONB} Retaining all assumptions in Theorem~\ref{thm:dirHaar}. In addition, if $c_B\le 2$ for each $B\in \cB_j$ and $j\in\N_0$, that is, the number of  children of each block $B$ is at most $2$, then $X(\{\cB_j\}_{j\in\N_0})$ is an orthonormal basis for $L_2(K)$.
\end{corollary}

\begin{proof}
Note that $L_2(K) =\overline{\cV_0\oplus_{j\in \N_0, B\in \cB_j} \cW_{j,B}}$ from the proof of Theorem~\ref{thm:dirHaar}.
If $c_B\le2$, then there is at most one element $\psi_{j,B}=\sqrt{b_2}\gamma_1-\sqrt{b_1}\gamma_2$ in $\Psi_{j,B}$ and $\|\psi_{j,B}\|_2=1$ for all $j,B$. Note that $\|\varphi_0\|_2=1$ also. Consequently, each $\cW_{j,B}$ is at most one-dimensional. Hence, $X(\{\cB_j\}_{j\in\N_0})$ is orthonormal. $\blacksquare$
\end{proof}

The Haar orthonormal wavelets on $I=[0,1]$ is a special case of the consequence of Corollary~\ref{cor:ONB}. It is with respect to the hierarchical partition $\{\cI_j\}_{j\in\N_0}$ of the unit interval $I$ with  $ \cI_j:=\{I_{j,k} := [2^{-j}k,2^{-j}(k+1)]\setsep k = 0,\ldots, 2^{j}-1\}$.  Note that each  $I_{j,k}$ has exactly two children subintervals $I_{j+1,2k}$ and  $I_{j+1,2k+1}$ with the same size. Such a type of partition is called \emph{dyadic}.  More generally,  in  dimension $d$,  we have the following result, which  includes directional Haar tight framelets in \cite{HLZ:AML,Li:DHF} as special cases.

\begin{corollary}
For $j\in\N_0$, let
$\cI_j:=\{I_{j,k} := [2^{-j}k,2^{-j}(k+1)]\setsep k = 0,\ldots, 2^{j}-1\}$
and define
\begin{equation}\label{B:dyadic}
\cB_j:=\otimes_d \cI_j:=\{I_{j,k_1}\times I_{j,k_2}\times \cdots\times I_{j,k_d}\setsep 0\le k_1,\ldots, k_d<2^{j}\}.
\end{equation}
Then, the system $X(\{\cB_j\}_{j\in\N_0})$  defined  as in Theorem~\ref{thm:dirHaar} is a tight frame for $L_2([0,1]^d)$. In particular, when $d=1$, it is the Haar orthonormal wavelets $X(I;\phi,\psi)$  defined  as in \eqref{D1Haar} and when $d=2$, it is the directional Haar tight framelets $X(I^2;\varphi,\psi)$ defined as in \eqref{D2DHF}.
\end{corollary}

\begin{proof}
It is easy to show that $\{\cB_j\}_{j\in\N_0}$ is a hierarchical partition of the unit block $I^d :=[0,1]^d$ in $d$-dimension. Each $B\in\cB_j$ has exactly $2^d$ children sub-blocks in $\cB_{j+1}$. The conclusions  follow directly from Theorem~\ref{thm:dirHaar} and Corollary~\ref{cor:ONB}.  $\blacksquare$
\end{proof}

The blocks in $\cB_j$  defined as in \eqref{B:dyadic} are dyadic. In practice, as demonstrated in Fig.~\ref{fig:coarse-grain-toy}, intervals used to identify nodes on a graph are not necessarily dyadic. Hence, we introduce the \emph{adaptive directional Haar tight framelets} (AdaDHF) based on the following hierarchical partition of the unit square $I^d=[0,1]^d$, whose sub-blocks are not necessarily dyadic.
\begin{corollary}
\label{cor:AdaDHF}
For each  $s=1,\ldots,d$, let $\{\cI_j^s\}_{j\in\N_0}$ be a hierarchical partition of the unit interval $I=[0,1]$, where
$
\cI_j^s:=\{I_{j,k}^s\setsep k = 1,\ldots, n_{j,s}\}.
$
Define
\begin{equation}\label{B:non-dya}
\cB_j:=\otimes_d \cI_j^s = \{I_{j,k_1}\times\cdots\times I_{j,k_d}\setsep k_s=1,\ldots, n_{j,s}, s = 1,\ldots, d\}.
\end{equation}
Then, $\{\cB_j\}_{j\in\N_0}$ is a hierarchical partition of the unit block $I^d=[0,1]^d$ and  the system $X(\{\cB_j\}_{j\in\N_0})$  defined  as in Theorem~\ref{thm:dirHaar} is a tight frame for $L_2([0,1]^d)$. In particular, for any $J\in\N_0$, the cut-off system $X(\{\cB_j\}_{j=0}^J)$ as defined in \eqref{def:cut-off} is a tight frame for $\cV_J$.
\end{corollary}
\begin{proof}
Since $\{\cI_j^s\}_{j\in\N_0}$ is a hierarchical partition of $I$, by the definition of $\cB_j$, $\{\cB_j\}_{j\in\N_0}$ is a hierarchical partition of the unit square $I^d$. The conclusions follow directly from Theorem~\ref{thm:dirHaar} and Corollary~\ref{cor:cut-off}.  $\blacksquare$
\end{proof}

\section{Digraph Signal Representations}
\label{sec:digraph}
In this section, we use the AdaDHF systems  developed in Corollary~\ref{cor:AdaDHF} of Section~\ref{sec:dirHaar} to investigate digraph signal representations. We study representations of signals on undirected graphs first and then turn to digraph signal representations.

\subsection{The coarse-grained chain of an undirected graph}
\label{subsec:coarse-grained-graph}
We start with undirected graphs and the coarse-grained chain of an undirected graph.  For an \emph{undirected graph} $\gph=(V,W)$ with vertex (or node) set $V=\{v_1,\ldots,v_n\}$ and adjacency matrix  $W: V\times V\rightarrow [0,\infty)$. We use $|V|$ (abuse of notation) to denote the number of vertices of $\gph$. The \emph{degree} of a vertex $v_i$ is denoted by $\dG(v_i):=\sum_{j=1}^n W(v_i,v_j)$. If $W(v_i,v_j)>0$, then it corresponds to an \emph{edge} $(v_i,v_j)$ (unordered pair). Two vertices $v_i, v_j$ are said to be \emph{connected} if there exists a \emph{path} between  them, that is, $[W^m](v_i,v_j)\neq0$ for some positive integer $m$. The graph $\gph$ is said to be \emph{connected} if there exists a path between any two vertices. Throughout the paper, we only consider \emph{connected} graphs.

Let $\gph=(V,W)$ and $\gph^{cg}=(V^{cg},W^{cg})$ be two undirected graphs. We say that $\gph^{cg}$ is a \emph{coarse-grained graph} of $\gph$ if $V^{cg}$ is a \emph{partition} of $V$; i.e., there exists subsets $U_1,\ldots,U_m$ of $V$ for some $m\in\N$ such that
\[
V^{cg} =\{U_1,\ldots,U_m\},\quad
U_1\cup\cdots\cup U_m = V, \quad
U_i\cap U_j = \emptyset,\quad 1\le i< j\le m.
\]
In such a case, each node $U_i$ of $\gph^{cg}$ is called a \emph{cluster}  from $\gph$. The edges of $\gph^{cg}$ are edges between clusters. Clusters  $U_1,\ldots, U_m$ define an \emph{equivalence relation} on $\gph$: two vertices $u,v\in\gph$ are equivalent, denoted by $u\sim v$, if $u$ and $v$ belong to the same cluster. An equivalent class (cluster) in $\gph$, which is a node in $\gph^{cg}$, associated with a vertex $v\in V$, then can be denoted as $[v]_{\gph^{cg}}:=\{u\in\gph \setsep u\sim v\}$, and we have $V^{cg}=V/{\sim} =\{[v]_{\gph^{cg}}\setsep v \in V\}$. If no confusion arises, we will drop the subscript $\gph^{cg}$ and simply use $[v]$ to denote a cluster from $\gph$. Note that a vertex $v$ in $\gph$ can be viewed as $[v]_\gph=\{v\}$, which is a singleton.

Given an undirected graph $\gph=(V,W)$, there are many clustering algorithms can be used to obtain clusters from $\gph$, see  e.g., \cite{CMZ:ACHA:18,arjuna2013,gavish2010multiscale,lafonncut,van2001graph}. Once we obtain the set $\{U_1,\ldots,U_m\}=:V^{cg}$ of clusters from $\gph$, we can define the weighted adjacency matrix $W^{cg}$ on $V^{cg}\times V^{cg}$ by
\begin{equation}\label{defn:wG:coarse}
W^{cg}([u],[v]):=\sum_{u\in[u]}\sum_{v\in[v]}{W(u,v)}, \quad [u],[v]\in V^{cg}.
\end{equation}
Then, the new graph $\gph^{cg}:=(V^{cg},W^{cg})$  is  a coarse-grained graph of $\gph$. Given the new graph $\gph^{cg}$, we can further apply clustering process on it and obtain a coarse-grained graph of $\gph^{cg}$. Recursively doing such clustering processes, we would obtain a \emph{chain} of graphs from the original graph $\gph$. More precisely, let $J\ge 0$ be an integer. We say that the sequence $\gph_{J\rightarrow 0}:=(\gph_J, \gph_{J-1},\ldots,\gph_{0})$ with $\gph_J\equiv \gph$ is a \emph{coarse-grained chain} of $\gph$  if $\gph_j=(V_j, W_j)$ is a coarse-grained graph of $\gph$ for all $0\le j\le J$ and $[v]_{\gph_j}\subseteq [v]_{\gph_{j-1}}$ for each $j=1,\ldots, J$ and for all $v\in V$. Note that, we treat each vertex $v$ of the finest level graph $\gph_J\equiv\gph$ as a cluster of singleton. See Fig.~\ref{fig:coarse-grain-toy} and Fig.~\ref{fig:coarse-grain-toy:Gy} for  illustrations of coarse-grained chains.

Once we have a coarse-grained chain $\gph_{J\rightarrow 0}$ of $\gph$, we next discuss how to associate it with a hierarchical partition of $I=[0,1]$.  Without loss of generality, we could assume that $\gph_0=(V_0,W_0)$ has only one node, which is a cluster consisting of all vertices of $\gph$. If not, we simply add such a graph to the chain.   Now we  define $\cI_j$ recursively as follows (c.f. Fig.~\ref{fig:coarse-grain-toy} and  Fig.~\ref{fig:coarse-grain-toy:Gy}).
\begin{itemize}
\item[{\rm 1)}] $\cI_0=\{I=[0,1]\}$ is the root node.

\item[{\rm 2)}] Suppose $\cI_{j-1}=\{I_{j-1,k}=[a_k,b_k]\setsep k=1,\ldots,|V_{j-1}|\}$ has been defined and associated with the graph $\gph_{j-1}=(V_{j-1},W_{j-1})$. Then, $I_{j-1,k}$ is associated with a node $u_k\in V_{j-1}$.

\item[{\rm 3)}] For each $u_k\in V_{j-1}$, denote (and order) the children of $u_k$ in $\gph_j=(V_j,W_j)$ as $u_{k,1},\ldots, u_{k,m}\in V_j$ and define subintervals $I_{j,k,1},\ldots, I_{j,k,m}$ by
\begin{equation}\label{def:I:split}
I_{j,k,s} = [a_k+w_{s-1},a_k+w_s],\quad s=1,\ldots, m,
\end{equation}
where $w_s:=(b_k-a_k)\times \frac{\sum_{i=1}^s \dG(u_{k,i})} {\sum_{i=1}^m \dG(u_{k,i})}$. Note that $[a_k,b_k]=\cup_s I_{j,k,s}$. Collect all  such subintervals $I_{j,k,s}$ as the collection   $\cI_j:=\{I_{j,k'}\setsep k'=1,\ldots, |V_j|\}$. Then $\cI_j$ is associated with the graph $\gph_j$.
\end{itemize}

Given a hierarchical partition $\{\cI_j\}_{j=0}^J$ that is associated with a coarse-grained chain of the graph $\gph=(V,W)$, then the vertex $v\in V$ is  associated with a subinterval $I_v\in \cI_{J}$.  A  signal $f:V\rightarrow \C$ on the graph can  be identified  as a function
\[
f = \sum_{v\in V} f(v) \chi_{I_v}
\]
defined on $I=[0,1]$.  We can thus define the space
\[
L_2(\gph):=L_2(\gph|\gph_{J\rightarrow0}):=\spn \{\chi_{I_v} \setsep v\in V\}\subseteq L_2([0,1])
\]
 with the usual norm $\|\cdot\|_2$ and inner product $\la\cdot,\cdot\ra$ for  $L_2([0,1])$. We immediately have the following result from Corollary~\ref{cor:cut-off}.

\begin{theorem}\label{thm:G}
The system $X(\{\cI_j\}_{j=0}^J)=\{\varphi_0\}\cup\{\Psi_{j,I}\setsep I\in \cI_j\}_{j=0}^{J-1}$  defined as in \eqref{def:cut-off} is a tight frame for $L_2(\gph)$.
\end{theorem}

We remark that, from Corollary~\ref{cor:ONB},  when each $I_{j,k}$ has at most two children,   such a system $X(\{\cI_j\}_{j=0}^J)$ is an orthonormal basis for $L_2(\gph)$ (c.f. \cite{treepap,CMZ:ACHA:18}).

\subsection{Digraph signal representations}
\label{subsec:digraph}

Now continue to the digraph case.
For a digraph $(V, W)$, the \textit{underlying undirected graph} is given by $(V, W_0)$, where $W_0=(W+W^\top)/2$.  A digraph  is \textit{weakly connected} if the underlying undirected graph is connected.  For simplicity, we restrict ourselves  in this paper to weakly connected digraph while results in the paper can be easily extended to general digraphs. To represent signals on the digraph $\gph$, we use the following steps to produce a pair $(\gph^x,\gph^y)$ of two undirected graphs:
\begin{itemize}
\item[{\rm 1)}] {\em Extension}: we define the \textit{extended graph} $\gph_e=(V, W_e)$ by  $W_e= I+W$, which is the same graph as $\gph=(V,W)$ except for a new (or enhanced) self-loop inserted at each vertex. This increases the connectivity of the undirected graphs obtained in the next step.

\item[{\rm 2)}] {\em Symmetrization}: define the \textit{pre-symmetrized graph}  $\gph_1=(V,W_1)$ and the \textit{post-symmetrized graph} (\cite{diagraphcluster_ohiostate_2011})  $\gph_2=(V, W_2)$ for the digraph $\gph_e=(V,W_e)$  by $W_1:=W_eW^\top_e$ and  $W_2:=W^\top_eW_e$.

\item[{\rm 3)}] {\em Post-processing}: remove the self--loops of $\gph_1$  and $\gph_2$ by $W^x:=W_1-\diag(W_1)$ and $W^y:=W_2-\diag(W_2)$.  Define $\gph^x = (V,W^x)$ and $\gph^y=(V,W^y)$.
\end{itemize}
It is not difficult to show that if $\gph$ is weakly connected, then $\gph^x$ and $\gph^y$ are connected (undirected) graphs.

We next use the pair $(\gph^x,\gph^y)$  to study signals defined on $\gph$.  As discussed in previous subsections, using various clustering algorithms, we can obtain coarse-grained chains $\gph_{J_x\rightarrow0}^x$ and $\gph^y_{J_y\rightarrow 0}$ of $\gph^x$ and $\gph^y$, respectively for some $J_x, J_y\in\N_0$. Without loss of generality, we can assume $J_x=J_y=:J$. In fact, if $J_x\neq J_y$, say $J_x<J_y$, then we  simply extend the chain $\gph^x_{J_x\rightarrow0}$ as $\gph^x_{J_y\rightarrow 0}$ by appending $\gph^x$:
\[
\gph^x_{J_y\rightarrow0}:=(\gph^x_{J_y},\ldots,\gph^x_{J_x+1}, \gph^x_{J_x},\ldots,\gph^x_0),
\]
where $\gph^x_{j}\equiv \gph^x$ for all $j\ge J_x$.

For each of the coarse-grained chains  $\gph^x_{J\rightarrow 0}$ and $\gph^y_{J\rightarrow 0}$, it is associated with a hierarchical partition $\{\cI_j^x\}_{j=0}^{J}$ and $\{\cI_j^y\}_{j=0}^{J}$, respectively.
Define
\begin{equation}\label{def:Bxy}
\cB_j:=\cI_j^x \otimes \cI_j^y:=\{I^x \times I^y \setsep  I^x\in \cI_j^x, I^y\in \cI_j^y \},\quad j=0,\ldots,J.
\end{equation}
Then, by Corollary~\ref{cor:AdaDHF}, we immediately have the following result.

\begin{theorem}\label{thm:AdaDHF:GxGy}  Let $\cB_j, j=0,\ldots,J$ be defined as in \eqref{def:Bxy} from  $\cI_j^x$ and $\cI_j^y$ associating with the coarse-grained chains $\gph^x_{J\rightarrow 0}$ and $\gph^y_{J\rightarrow0}$ for graphs $\gph^x, \gph^y$, respectively.  Then, the system
 $X(\{\cB_j\}_{j=0}^J)$  defined as in \eqref{def:cut-off} is a tight frame for
 $L_2(\gph^x,\gph^y):=L_2(\gph^x|\gph^x_{J\rightarrow0},\gph^y|\gph^y_{J\rightarrow0}):=\cV_J=\spn\{\chi_B\setsep B\in \cB_J\}$.
\end{theorem}

\medskip

For signals $f:V\rightarrow \C$ defined on the digraph $\gph$, we can define the digraph signal space $L_2(\gph)$ as follows. For $v\in V$, there are $I_v^x \in \cI_J^x$ and $I_v^y\in \cI_J^y$. $B_v:=I_v^x\times I_v^y$ is then a block in $\cB_J$. Thus, $f$ can be identified as a function defined on $[0,1]^2$:
\begin{equation}\label{def:f:G}
f = \sum_{v\in V} f(v) \chi_{B_v},\quad B_v = I_v^x\times I^y_v, \, v\in V.
\end{equation}
Hence, we can define $L_2(\gph)$ as
\[
L_2(\gph):=L_2(\gph|(\gph^x_{J\rightarrow0},\gph^y_{J\rightarrow0})):=\spn\{\chi_{B_v} \setsep v\in V\}.
\]
with the usual norm $\|\cdot\|_2$ and inner product $\la\cdot,\cdot\ra$ for  $L_2([0,1]^2)$.

Since $f\in L_2(\gph)$ is supported on $\cup_{v\in V} B_v$,  we can conclude this section by the following result.
\begin{corollary}\label{cor:AdaDHF:G}
Let $X(\{\cB_j\}_{j=0}^J)$ be defined as in Theorem~\ref{thm:AdaDHF:GxGy} and $B_v, v\in V$ be blocks defined as in \eqref{def:f:G} associated with the digraph $\gph$. Define
\begin{equation}\label{X:G}
\begin{aligned}
X(\{\cB_j\}_{j=0}^J|_\gph)&:=\{\varphi_0\}\cup\\
&\{\psi\setsep |\supp \psi \cap B_v|>0 \mbox{ for some } v\in V, \psi\in\Psi_{j,B}, B\in\cB_j\}_{j=0}^{J-1}.
\end{aligned}
\end{equation}
Then $X(\{\cB_j\}_{j=0}^J|_\gph)$ is a tight frame for $L_2(\gph)$.
\end{corollary}
\begin{proof}
Note that
\[
L_2(\gph)\subseteq L_2(\gph^x,\gph^y)\subseteq L_2([0,1]^2).
\]
Hence, by Theorem~\ref{thm:AdaDHF:GxGy}, any $f\in L_2(\gph)$ can be represented by the tight frame system $X(\{\cB_j\}_{j=0}^N)$ as
\[
f = \la f, \varphi_0\ra \varphi_0+\sum_{j=0}^{J-1} \sum_{B\in\cB_j}\sum_{1\le \ell_1<\ell_2\le c_B} \la f, \psi_{j,B}^{(\ell_1,\ell_2)}\ra \psi_{j,B}^{(\ell_1,\ell_2)}.
\]
Since $f$ is supported on $\cup_{v\in V}B_v$, we can discard those of $\psi^{(\ell_1,\ell_2)}_{j,B}$ whose essential support  is not intersecting with any  $B_v, v\in V$. Then,
\[
f = \la f, \varphi_0\ra \varphi_0+\sum_{j=0}^{J-1} \sum_{B\in\cB_j}\sum_{1\le \ell_1<\ell_2\le c_B,|\supp \psi\cap B_v|\neq 0} \la f, \psi_{j,B}^{(\ell_1,\ell_2)}\ra \psi_{j,B}^{(\ell_1,\ell_2)}.
\]
That is, the restriction  $X(\{\cB_j\}_{j=0}^J|_\gph)$ of $X(\{\cB_j\}_{j=0}^J)$ on $\gph$ is a tight frame for $L_2(\gph)$.
This completes the proof. $\blacksquare$
\end{proof}

\section{Illustration Examples}
\label{sec:example}
In this section, we provide some examples to illustrate results in previous sections.

\subsection{Example 1: a tight frame on an  undirected graph}
Let $\gph=(V,W)=\gph_3$ be the graph in Fig.~\ref{fig:coarse-grain-toy} (see also $\gph^x$ in Fig.~\ref{fig:G:Gx:Gy}). That is,
\begin{equation}
\label{def:ex1:G}
W =
\left[
\begin{matrix}
0 & 1 &1 & 0 &0 & 0   \\
1 & 0 &0 & 0 &0 & 0   \\
1 & 0 &0 & 1 &1 & 1   \\
0 & 0 &1 & 0 &1 & 0   \\
0 & 0 &1 & 1 &0 & 0   \\
0 & 0 &1 & 0 &0 & 0   \\
\end{matrix}
\right],
\end{equation}
where the rows and columns are ordered from 1 to 6 with respect to the vertices $a,b,c,d,e,f$ in $V$.

Applying clustering algorithms, e.g., the NHC algorithm in \cite{CMZ:ACHA:18}, to the graph $\gph$,  we can obtain a coarse-grained chain $\gph_{3\rightarrow0}:=(\gph_3,\ldots,\gph_0)$ of $\gph$ as in Fig.~\ref{fig:coarse-grain-toy}. Each vertex in $\gph_{j-1}$ is a cluster of vertices in $\gph_{j}$. Based on such a coarse-grained chain, we can give a hierarchical sequence $\{\cI_j\}_{j=0}^3$ and  build our tight frame system $X(\{\cI_j\}_{j=0}^3)=\{\varphi_0\}\cup\{\Psi_j\}_{j=0}^{2}$ as follows (see Fig.~\ref{fig:coarse-grain-toy}).
\begin{itemize}
\item[{1)}] The root node ($j=0$):  $\gph_0\longleftrightarrow\cI_0:=\{I_{0,1}=I=[0,1]\}$. This is associated with $\varphi_0 = \chi_{[0,1]}$.

\item [{2)}] At level $j=1$:  $\gph_1\longleftrightarrow\cI_1:=\{I_{1,1}, I_{1,2}\}$. The graph $\gph_1$ has two nodes $[a]_{\gph_1}=\{a,b\}$ (degree 3) and $[c]_{\gph_1}=\{c,d,e,f\}$ (degree 9). According to their degrees and \eqref{def:I:split}, we  identify the nodes $[a]_{\gph_1}$ and $[c]_{\gph_1}$ with the intervals
\begin{equation}\label{Gx:Ix:1}
I_{1,1}([a])=\left[0,\frac{3}{12}\right)=\left[0,\frac14\right)\mbox{ and } I_{1,2}([c])=\left[\frac{3}{12},\frac{3+9}{12}\right]=\left[\frac14,1\right],
\end{equation}
respectively. The two subintervals $I_{1,1}$ and $I_{1,2}$ are the two children of $I_{0,1}$. Hence, by \eqref{def:psijb}, $\Psi_0=\{\psi_0\}$ with
\[
\psi_0:=\frac{3\chi_{I_{1,1}}-\chi_{I_{1,2}}}{\sqrt{3}}.
\]
Note that $\|\psi_0\|_2 = 1$ and $\|\psi_0\|_1=0$.

\item [{3)}]At level $j=2$: $\gph_2\longleftrightarrow\cI_2:=\{I_{2,1},I_{2,2},I_{2,3}\}$. The graph $\gph_2$ has three nodes $[a]_{\gph_2}=\{a,b\}$ (degree 3), $[c]_{\gph_2}=\{c,d,e\}$ (degree 8), and $[f]_{\gph_2}=\{f\}$ (degree 1). Similarly, we  identify them  with the intervals
\begin{equation}\label{Gx:Ix:2}
I_{2,1}([a])=\left[0,\frac14\right), I_{2,2}([c])=\left[\frac14,\frac{11}{12}\right), I_{2,3}([f])=\left[\frac{11}{12},1\right],
\end{equation}
respectively. Only $I_{2,2}$ and $I_{2,3}$ are split from $I_{1,2}$. Hence,  by \eqref{def:psijb}, $\Psi_1=\{\psi_1\}$ with
\[
\psi_1:=\frac{\chi_{I_{2,2}}-8\chi_{I_{2,3}}}{\sqrt{6}}.
\]
Note that $\|\psi_1\|_2 = 1$ and $\|\psi_1\|_1=0$.

\item [{4)}] At level $j=3$: $\gph_3\longleftrightarrow\cI_3:=\{I_{3,1},\ldots, I_{3,6}\}$. The graph $\gph_3$ is the underlying graph with 6 vertices. According to their degrees, we  identify the vertices $a,b,c,d,e,f$ with the intervals
\begin{equation}\label{Gx:vertices}
\begin{aligned}
I_{3,1}(a)&=\left[0,\frac16\right),&
I_{3,2}(b)&=\left[\frac16,\frac14\right),&
I_{3,3}(c)&=\left[\frac14,\frac{7}{12}\right),&\\
I_{3,4}(d)&=\left[\frac{7}{12},\frac{3}{4}\right),&
I_{3,5}(e)&=\left[\frac{3}{4},\frac{11}{12}\right),&
I_{3,6}(f) &= \left[\frac{11}{12},1\right],
\end{aligned}
\end{equation}
respectively. Note that $I_{2,1}$ is split to $I_{3,1}$ and $I_{3,2}$ while $I_{2,2}$ is split to $I_{3,3}, I_{3,4}, I_{3,5}$. Hence,  by \eqref{def:psijb}, $\Psi_2=\{\psi_2,\psi_3,\psi_4,\psi_5\}$ with
\[
\begin{aligned}
&\psi_2:=\sqrt{2}({\chi_{I_{3,1}}-2\chi_{I_{3,2}}}), &
\psi_3:=\frac{\sqrt{3}}{2}(\chi_{I_{3,3}}-2\chi_{I_{3,4}}),\\
&\psi_4:=\frac{\sqrt{3}}{2}(\chi_{I_{3,3}}-2\chi_{I_{3,5}}), &
\psi_5:=\sqrt{\frac{3}{2}}(\chi_{I_{3,4}}-\chi_{I_{3,5}}).
\end{aligned}
\]
Note that $\|\psi_i\|_1 = 0$ for $i=2,\ldots, 5$.
\end{itemize}

\medskip

By Theorem~\ref{thm:G}, $X(\{\cI_j\}_{j=0}^3) = \{\varphi_0,\psi_0,\ldots,\psi_5\}$ is a tight frame for  $L_2(\gph)=\spn\{\chi_{I_{3,k}}\setsep k = 1,\ldots,6\}$.

\subsection{Example 2: tight frames on a digraph}

Let the digraph $\gph=(V,W)$ be determined by
\[
W =
\left[
\begin{matrix}
0 & 1 & 0 & 0 & 0 & 0\\
0 & 0 & 0 & 0 & 0 & 0\\
1 & 0 & 0 & 1 & 0 & 1\\
0 & 0 & 0 & 0 & 1 & 0\\
0 & 0 & 1 & 0 & 0 & 0\\
0 & 0 & 0 & 0 & 0 & 0\\
\end{matrix}
\right],
\]
where the rows and columns are ordered from 1 to 6 with respect to the vertices $a,b,c,d,e,f$ in $V$, see Fig.~\ref{fig:G:Gx:Gy}.  After applying the symmetrization processing as discussed in Section~\ref{subsec:digraph}, we obtain two undirected graphs $\gph^x = (V,W^x)$ and $\gph^y=(V,W^y)$, where $\gph^x$ is the same graph considered as in Example~1 (see its adjacency matrix in \eqref{def:ex1:G}) and $\gph^y$ is determined by $W^y$:
\[
W^y=
\left[
\begin{matrix}
0 & 1 & 1 & 1 & 0 & 1\\
1 & 0 & 0 & 0 & 0 & 0\\
1 & 0 & 0 & 1 & 1 & 1\\
1 & 0 & 1 & 0 & 1 & 1\\
0 & 0 & 1 & 1 & 0 & 0\\
1 & 0 & 1 & 1 & 0 & 0\\
\end{matrix}
\right].
\]

\begin{figure}[htpb!]
\begin{minipage}{\textwidth}
\begin{center}
\begin{tikzpicture}[
roundnode/.style={circle, draw=green!60, fill=green!5, very thick, minimum size=7mm},
squarednode/.style={rectangle, draw=red!60, fill=red!5, very thick, minimum size=3mm},
]

\node[squarednode]      (v1)       at (-2.5,0)      {$a$};
\node[squarednode]      (v2)       at (-1.5,0)      {$b$};
\node[squarednode]      (v3)       at (0,0)         {$c$};
\node[squarednode]      (v4)       at (1,0)         {$d$};
\node[squarednode]      (v5)       at (2,0)         {$e$};
\node[squarednode]      (v6)       at (3.5,0)       {$f$};

\draw[->] (v1) -- (v2); \draw[]  (-2.0,0.2) node{1};
\draw[->] (v3) -- (v4);  \draw[]  (0.5,0.2) node{1};
\draw[->] (v4) -- (v5); \draw[]   (1.5,0.2) node{1};
\draw[<-] (v1.south).. controls (-2,-0.8) and (-0.5,-0.8) .. (v3.south);\draw (-1.1,-0.85) node{1};
\draw[<-] (v3.north).. controls (0.5,0.8) and (1.5,0.8)   .. (v5.north); \draw[] (1,0.85)     node{1};
\draw[->] (v3.south).. controls (0.8,-0.8) and (2.4,-0.8) .. (v6.south); \draw[] (1.9,-0.85)  node{1};
\draw[]  (5,0) node[circle,draw](g2){$\gph$};

\node[squarednode]      (v1)       at (-2.5,2)      {$a$};
\node[squarednode]      (v2)       at (-1.5,2)      {$b$};
\node[squarednode]      (v3)       at (0,2)         {$c$};
\node[squarednode]      (v4)       at (1,2)         {$d$};
\node[squarednode]      (v5)       at (2,2)         {$e$};
\node[squarednode]      (v6)       at (3.5,2)       {$f$};

\draw[-] (v1) -- (v2)  (-2.0,2.2) node{1};
\draw[-] (v3) -- (v4)   (0.5,2.2) node{1};
\draw[-] (v4) -- (v5)   (1.5,2.2) node{1};
\draw[-] (v1.south).. controls (-2,1.2) and (-0.5,1.2) .. (v3.south)  (-1.1,1.15) node{1};
\draw[-] (v3.north).. controls (0.5,2.8) and (1.5,2.8)   .. (v5.north)  (1,2.85)     node{1};
\draw[-] (v3.south).. controls (0.8,1.2) and (2.4,1.2) .. (v6.south)  (1.9,1.15)  node{1};
\draw[]  (5,2) node[circle,draw](g1){$\gph^x$};

\node[squarednode]      (v1)       at (-2.5,4)      {$a$};
\node[squarednode]      (v2)       at (-1.5,4)      {$b$};
\node[squarednode]      (v3)       at (0,4)         {$c$};
\node[squarednode]      (v4)       at (1,4)         {$d$};
\node[squarednode]      (v5)       at (2,4)         {$e$};
\node[squarednode]      (v6)       at (3.5,4)       {$f$};

\draw[-] (v1) -- (v2)  (-2.0,4.2) node{1};
\draw[-] (v1.south).. controls (-2,3.2) and (-0.5,3.2) .. (v3.south)  (-1.1,3.15) node{1};
\draw[-] (v1.north).. controls (-2,4.8) and (0.5,4.8) .. (v4.north)  (-0.9,4.8) node{1};
\draw[-] (v1.north).. controls (-2,5.4) and (2.5,5.4) .. (v6.north)  (0.3,5.3) node{1};

\draw[-] (v3) -- (v4)   (0.5,4.2) node{1};
\draw[-] (v3.north).. controls (0.5,4.8) and (1.5,4.8)   .. (v5.north)  (1,4.85)     node{1};
\draw[-] (v3.south).. controls (0.5,3.1) and (3,3.1) .. (v6.south)  (1.9,3.15)  node{1};

\draw[-] (v4) -- (v5)   (1.5,4.2) node{1};
\draw[-] (v4.south).. controls (1.5,3.4) and (3,3.4) .. (v6.south)  (2.3,3.65) node{1};
\draw[]  (5,4) node[circle,draw](g1){$\gph^y$};

\end{tikzpicture}
\end{center}
\end{minipage}\vspace{2mm}
\caption{Symmetrization of $\gph$ (bottom) gives a pair $(\gph^x,\gph^y)$ of undirected graphs (top and middle).}
\label{fig:G:Gx:Gy}
\end{figure}
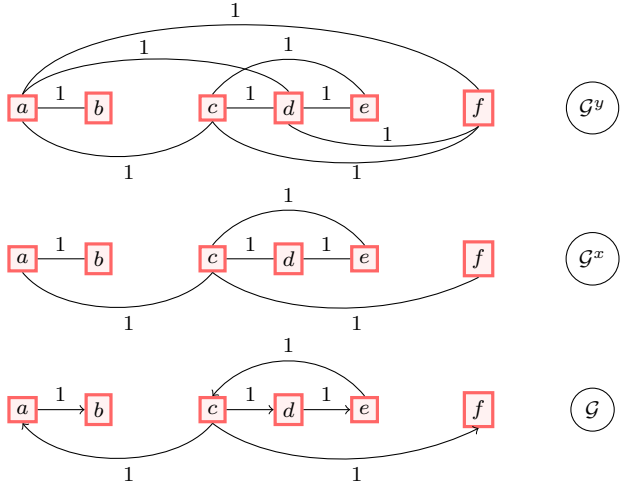

Applying clustering algorithms, e.g., the NHC algorithm in \cite{CMZ:ACHA:18}, to the undirected graphs $\gph^x, \gph^y$, we can obtain  coarse-grained chains $\gph^x_{3\rightarrow0}:=(\gph^x_3,\ldots,\gph^x_0)$ and $\gph^y_{3\rightarrow0}:=(\gph^y_3,\ldots,\gph^y_0)$ of $\gph^x$ and  $\gph^y$, respectively. The coarse-grained chain $\gph^x$ and its associated interval collection $\{\cI_j^x\}_{j=0}^3$ are already shown in Example~1 (or see Fig.~\ref{fig:coarse-grain-toy}).   Now we describe the coarse-grained chain $\gph^y_{3\rightarrow0}$ and its associated interval collections $\{\cI_j^y\}_{j=0}^3$ (see Fig.~\ref{fig:coarse-grain-toy:Gy}).
Based on the hierarchical interval sequences $\{\cI^x_j,\cI_j^y\}_{j=0}^3$, we can build the hierarchical block sequence $\{\cB_j=\cI^x_j\otimes \cI^y_j\}_{j=0}^3$ defined as in \eqref{def:Bxy}. See Fig.~\ref{fig:GxGyG} (top) for the illustration of blocks and the blocks in $\cB_j$ are labelled with a number from 0 to 49 in Fig.~\ref{fig:GxGyG} (bottom).

\begin{itemize}
\item[{1)}] The root node ($j=0$):  $\gph_0^y\longleftrightarrow\cI_0^y:=\{I^y_{0,1}=I=[0,1]\}$. Together with $\cI_0^x=\{[0,1]\}$, we have $\cB_0=\{I^2=[0,1]^2\}$ and $\varphi_0 = \chi_{[0,1]^2}$.
\item [{2)}] At level $j=1$:  $\gph_1^y\longleftrightarrow\cI_1^y:=\{I^y_{1,1}, I^y_{1,2}\}$. The graph $\gph_1^y$ has two nodes $[a]_{\gph_1^y}=\{a,b,d\}$ (degree 9) and $[c]_{\gph_1^y}=\{c,e,f\}$ (degree 9). According to their degrees, we  identify the nodes $[a]_{\gph_1^y}$ and $[c]_{\gph_1^y}$  with the intervals
\[
I^y_{1,1}([a])=\left[0,\frac{9}{18}\right)=\left[0,\frac12\right) \mbox{ and } I^y_{1,2}([c])=\left[\frac{9}{18},\frac{9+9}{18}\right]=\left[\frac12,1\right],
\]
respectively. From $\cI_1^x=\{I_{1,1}^x,I_{1,2}^x\}$, see \eqref{Gx:Ix:1}, $\cB_1=\cI_1^x\otimes\cI_1^y$ has four sub-blocks $B_{1},\ldots,B_{4}$ in $[0,1]^2$. Hence,  by \eqref{def:psijb}, $\Psi_1=\{\psi^{(\ell_1,\ell_2)}_{0,[0,1]^2}\setsep 1\le \ell_1,\ell_2\le 4\}$ has 6 functions.

\item [{3)}]At level $j=2$: $\gph_2^y\longleftrightarrow\cI^y_2:=\{I^y_{2,1},I^y_{2,2},I^y_{2,3}\}$. The graph $\gph^y_2$ has three nodes $[a]_{\gph_2^y}=\{a,b,d\}$ (degree 9), $[c]_{\gph^y_2}=\{c,e\}$ (degree 6), and $[f]_{\gph^y_2}=\{f\}$ (degree 3). Similarly, we  identify them  with the intervals
\[
I^y_{2,1}([a])=\left[0,\frac{1}{2}\right),\, I^y_{2,2}([c])=\left[\frac{1}{2},\frac{5}{6}\right), \,I^y_{2,3}([f])=\left[\frac{5}{6},1\right],
\]
respectively. From $\cI_2^x$ in \eqref{Gx:Ix:2}, $\cB_2=\cI_2^x\otimes\cI_2^y=\{B_5,\ldots,B_{13}\}$ has nine sub-blocks split from $B_1,\ldots,B_4$: $B_{1}$ has  $c_{B_{1}}=1
$ sub-block $B_{5}=B_{1}$; $B_{2}$ has $c_{B_{2}}=2$ sub-blocks $B_{6}$ and  $B_{7}$;$B_{3}$  has $c_{B_{3}}=2$ sub-blocks $B_{8}$ and $B_{11}$;   $B_{4}$ has $c_{B_{4}}=4$ sub-blocks $B_{9}$, $B_{10}$, $B_{12}$, and  $B_{13}$.  Hence, by \eqref{def:psijb}, $\Psi_1=\{\psi^{(\ell_1,\ell_2)}_{1,B_{k}}\setsep 1\le \ell_1,\ell_2\le c_{B_{k}},k=1,2,3,4\}$ has in total ${2\choose 2}+{2\choose 2}+{4\choose 2}=8$ functions.

\item [{4)}] At level $j=3$: $\gph_3^y\longleftrightarrow\cI^y_3:=\{I^y_{3,1},\ldots, I^y_{3,6}\}$. The graph $\gph^y_3$ is the underlying graph $\gph^y$ with 6 vertices. According to their degrees, we  identify the vertices $a,b,d,c,e,f$ with the intervals

\begin{equation}\label{Gy:vertices}
\begin{aligned}
I^y_{3,1}(a)&=\left[0,\frac{2}{9}\right), &
I^y_{3,2}(b)&=\left[\frac{2}{9},\frac{5}{18}\right), &
I^y_{3,3}(d)&=\left[\frac{5}{18},\frac{1}{2}\right), &\\
I^y_{3,4}(c)&=\left[\frac{1}{2},\frac{13}{18}\right),&
I^y_{3,5}(e)&=\left[\frac{13}{18},\frac{5}{6}\right),&
I^y_{3,6}(f) &= \left[\frac{5}{6},1\right],&
\end{aligned}
\end{equation}
respectively. From $\cI_3^x$ in \eqref{Gx:vertices}, $\cB_3=\cI_3^x\otimes \cI_3^y$ has in total 36 blocks $B_{14}$, $\ldots$, $B_{49}$. Hence, by \eqref{def:psijb}, $\Psi_2=\{\psi^{(\ell_1,\ell_2)}_{2,B_{k}}\setsep 1\le \ell_1,\ell_2\le c_{B_{k}},k=5,\ldots,13\}$ containing ${6 \choose 2}+{9\choose 2}+{3 \choose 2}+{4 \choose 2}+{6 \choose 2}+{2 \choose 2}+{2\choose 2}+{3 \choose 2}=15+36+3+6+15+1+1+3=80$ functions.
\end{itemize}
According to Theorem~\ref{thm:AdaDHF:GxGy}, the  system $X(\{\cB_j\}_{j=0}^3)=\{\varphi_0\}\cup\{\Psi_j\}_{j=0}^{2}$ is a tight frame for $L_2(\gph^x,\gph^y)$.

\begin{figure}[htpb!]
\begin{minipage}{\textwidth}
\centering
\begin{minipage}{\textwidth}
\begin{center}
\begin{tikzpicture}[
roundnode/.style={circle, draw=green!60, fill=green!5, very thick, minimum size=7mm},
squarednode/.style={rectangle, draw=red!60, fill=red!5, very thick, minimum size=3mm},
]

\node[squarednode]      (v1)       at (-2.5,-0.3)      {$a$};
\node[squarednode]      (v2)       at (-1.5,-0.3)      {$b$};
\node[squarednode]      (v3)       at (-0.7,-0.3)         {$d$};
\node[squarednode]      (v4)       at (0.7,-0.3)         {$c$};
\node[squarednode]      (v5)       at (1.7,-0.3)         {$e$};
\node[squarednode]      (v6)       at (3.5,-0.3)       {$f$};

\draw[-] (v1) -- (v2)  (-2.0,-0.15) node{1};
\draw[-] (v1.south).. controls (-2.3,-1.1) and (-0.8,-1.1) .. (v3.south)  (-1.5,-0.95) node{1};
\draw[-] (v1.north).. controls (-2,0.5) and (0.5,0.5) .. (v4.north)  (-0.9,0.4) node{1};
\draw[-] (v1.north).. controls (-2,1.1) and (2.5,1.1) .. (v6.north)  (0.35,0.8) node{1};

\draw[-] (v3) -- (v4)   (0.1,-0.15) node{1};
\draw[-] (v3.north).. controls (0.1,0.5) and (1.1,0.5)   .. (v5.north)  (0.6,0.35)     node{1};
\draw[-] (v3.south).. controls (0.5,-1.2) and (3,-1.2) .. (v6.south)  (1.9,-1.05)  node{1};

\draw[-] (v4) -- (v5)   (1.2,-0.15) node{1};
\draw[-] (v4.south).. controls (1.5,-0.9) and (3,-0.9) .. (v6.south)  (2.9,-0.72) node{1};
\draw[]  (5,-0.3) node[circle,draw](g3){$\gph_3^y$};

\draw[]  (-2.4,0.2) node{$[0,\frac{2}{9})$};
\draw[]  (-1.5,0.2) node{$[\frac{2}{9},\frac{5}{18})$};
\draw[]  (-0.6,-0.7) node{$[\frac{5}{18},\frac{1}{2})$};
\draw[]  (0.7,-0.7) node{$[\frac{1}{2},\frac{13}{18})$};
\draw[]  (1.9,-0.7) node{$[\frac{13}{18},\frac{5}{6})$};
\draw[]  (3.3,0.2) node{$[\frac{5}{6},1]$};

\node[squarednode]      (v7)       at (-1.7,2.0)        {$a$ \qquad \;$b$\qquad \;$d$};
\node[squarednode]      (v8)       at (1,2.0)            {$c$ \qquad $e$};
\node[squarednode]      (v9)       at (3.5,2.0)         {$f$};

\draw[<->] (-2.5,2.2) arc (150:30:9mm);
\draw (-1.7,2.5) node{4};
\draw[-] (v7) -- (v8)    (-0.1,2.2) node{3};
\draw[-] (v7.south).. controls (0.3,1.0) and (1.5,1.0)   .. (v9.south)  (1,1.2)     node{2};

\draw[-] (v8) -- (v9)    (2.45,2.2)  node{1};
\draw[<->] (0.5,2.2)  arc (150:30:6mm);
\draw (1,2.35)  node{2};

\draw[]  (5,2.0) node[circle,draw](g2){$\gph^y_2$};
\draw[]  (-2.0,1.6) node{$[0,\frac12)$};
\draw[]  (1.0,1.6) node{$[\frac12,\frac{5}{6})$};
\draw[]  (3.3,1.6) node{$[\frac{5}{6},1]$};

\node[squarednode]      (v10)       at (-1.7,3.4)        {$a$ \qquad \;$b$\qquad \;$d$};
\node[squarednode]      (v11)       at (2.0,3.4)         {$c$\qquad \qquad \;\;$e$\qquad\qquad\;$f$};

\draw[<->] (-2.5,3.6) arc (150:30:9mm);
\draw (-1.7,3.9) node{4};
\draw[-] (v10) -- (v11)    (-0.2,3.6) node{5};

\draw[<->] (0.4,3.7)  arc (120:60:31mm);
\draw (1.9,4.0)  node{4};
\draw[]  (5,3.4) node[circle,draw](g1){$\gph^y_1$};
\draw[]  (-2.0,3.0) node{$[0,\frac12)$};
\draw[]  (1.75,3.0) node{$[\frac12,1]$};

\node[squarednode]      (v12)       at (0.45,4.7)         {$a$\qquad\quad$b$\qquad\quad$d$\qquad\qquad\qquad$c$\quad\qquad $e$\quad\qquad $f$};
\draw[]  (5,4.7) node[circle,draw](g0){$\gph^y_0$};
\draw[]  (0,4.3) node{$[0,1]$};

\draw[->] (g3.north) -- (g2.south);
\draw[->] (g2.north) -- (g1.south);
\draw[->] (g1.north) -- (g0.south);

\end{tikzpicture}
\end{center}
\end{minipage}\vspace{2mm}
\begin{minipage}{0.8\textwidth}
\caption{A coarse-grained chain of $\gph^y$. $\gph^y_3$ is the underlying graph $\gph^y$. $\gph^y_{j-1}$ is from clustering of $\gph^y_j$ for $j=1,2,3$. Note that $\gph^y_0$ has one node only. Here each box represents a node (or cluster) in the graph, the lines represent edges between vertices, and the arc on a same node indicates a self-loop. $\gph^y_0$ can be identified as the root interval $I^y=[0,1]$, $\gph^y_1$ as $[0,\frac12)\cup [\frac12,1]$, $\gph^y_2$ as $[0,\frac12)\cup[\frac12,\frac{5}{6})\cup[\frac{5}{6},1]$, and $\gph^y_3$ as $[0,\frac{2}{9})\cup[\frac{2}{9},\frac{5}{18})\cup[\frac{5}{18},\frac{1}{2})\cup[\frac{1}{2},\frac{13}{18})\cup[\frac{13}{18},\frac{5}{6})\cup[\frac{5}{6},1]$.}
\label{fig:coarse-grain-toy:Gy}
\end{minipage}
\end{minipage}
\end{figure}

\begin{figure}[htbp!]
\centering
\includegraphics[width=4.6in]{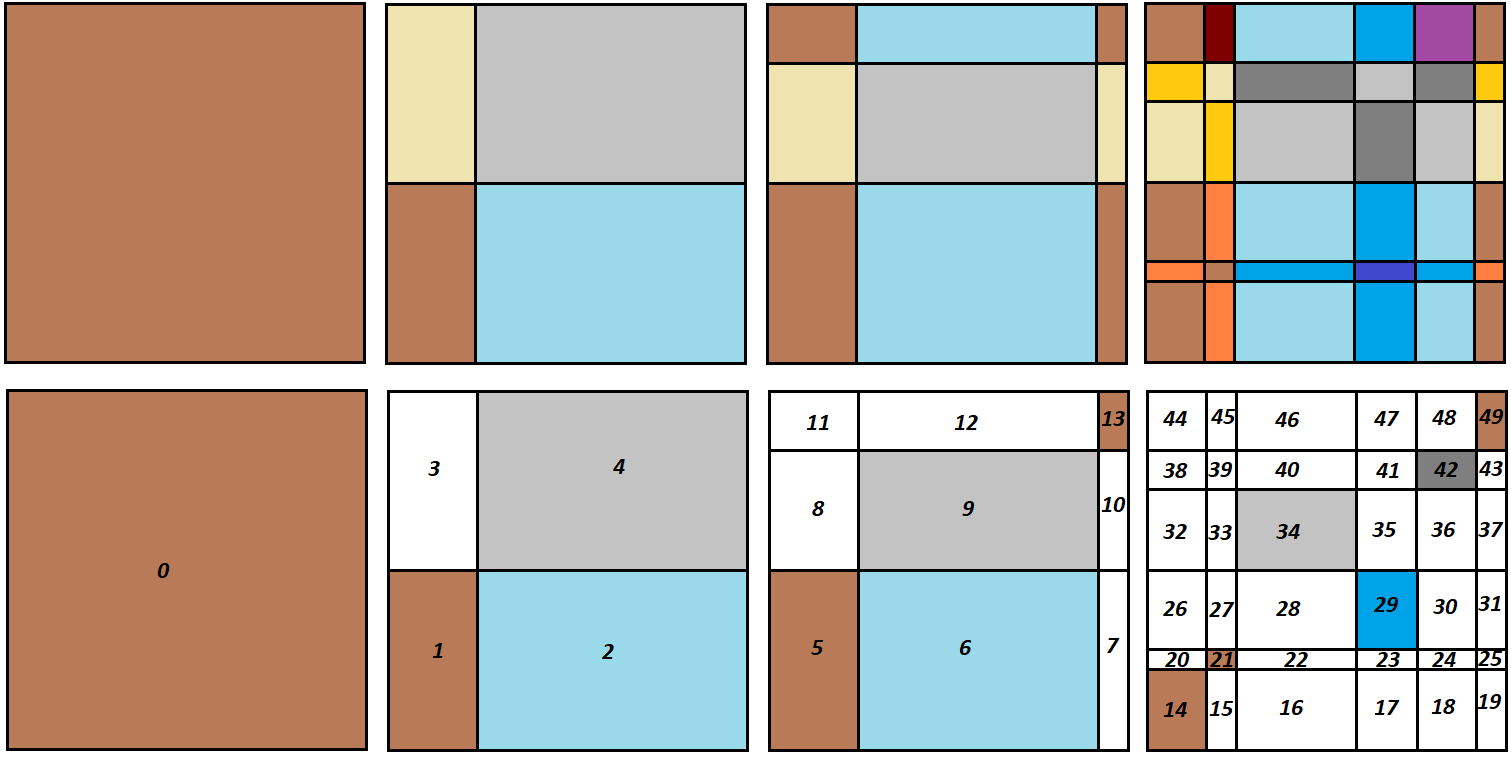}
\caption{Each big block is the unit square $[0,1]^2$ (totally 8). Top 4 unit squares (left to right): $\cB_0, \cB_1,\cB_2,\cB_3$ from the coarse-grained chains $\gph^x_{3\rightarrow0}$ ($\cI^x_j, j=0,\ldots,3$) and $\gph^y_{3\rightarrow0}$ ($\cI^y_j, j=0,\ldots,3$) in Example~2. Horizontal axis is $x$ while the vertical axis is $y$.
 Each colored sub-block represents $B=I_x\times I_y$ for some  $I_x\in \cI^x_j$ and $I_y\in \cI_j^y$. Note that the top-right square contains 36 sub-blocks from $\cI^x_3\otimes \cI_3^y$ with respect to $V\times V$ in $\gph^x=(V,W^x)$ and $\gph^y=(V,W^y)$. Bottom 4 unit squares: sub-blocks in each of $\cB_j$  are selected so that $|B\cap B_v|\neq0$ for some $v\in V$. White blocks are those discarded ones.  This is with respect to the system $X(\{\cB_j\}_{j=0}^3|_\gph)$. Each block in the unit square is labelled with a number from 0 to 49. The blocks $B_{14},B_{21}, B_{34},B_{29}, B_{42}, B_{49}$ are the vertices $a, b, c, d,e, f$ in the digraph, respectively. }
\label{fig:GxGyG}
\end{figure}

Next, we focus on the tight frame $X(\{\cB_j\}_{j=0}^3|_\gph)$ for $L_2(\gph)$ as in Corollary~\ref{cor:AdaDHF:G}. Note that the vertices $a,b,c,d,e,f$ in the  digraph $\gph$ are with respect to blocks
\begin{equation}\label{G:vertices}
\begin{small}
\begin{aligned}
B_a &= \left[0,\frac16\right)\times \left[0,\frac29\right),&
B_b &= \left[\frac16,\frac14\right)\times \left[\frac29,\frac{5}{18}\right), &
B_c &= \left[\frac14,\frac{7}{12}\right)\times \left[\frac12,\frac{13}{18}\right),&  \\
B_d &= \left[\frac{7}{12}, \frac34\right)\times \left[\frac{5}{18},\frac12\right),&
B_e &= \left[\frac34,\frac{11}{12}\right)\times \left[\frac{13}{18},\frac56\right),&
B_f &= \left[\frac{11}{12},1\right]\times \left[\frac56,1\right]. &
\end{aligned}
\end{small}
\end{equation}
The functions in  $X(\{\cB_j\}_{j=0}^3|_\gph)$ are as follows.
\begin{itemize}
\item[{\rm a)}] Supported on $\cB_0$: $\{\varphi_0\}\cup\{\psi_{0,[0,1]^2}^{(\ell_1,\ell_2)}\setsep 1\le\ell_1<\ell_2\le 4\}$. There are $7=1+6$ functions with support intersecting $B_v$  in \eqref{G:vertices}.

\item[{\rm b)}] Supported on $\cB_1$:  $\psi^{(\ell_1,\ell_2)}_{1,B_k}$ for $k=2,4$ and $1\le \ell_1<\ell_2<c_{B_k}$. There are $6=1+5$ functions with support intersecting $B_v$  in \eqref{G:vertices}.

\item[{\rm c)}] Supported on $\cB_2$: $\psi^{(\ell_1,\ell_2)}_{2,B_k}$ for $k=5,6,9$ and $1\le \ell_1<\ell_2<c_{B_k}$. There are $26=9+8+9$ functions with support intersecting $B_v$ in \eqref{G:vertices}.

\end{itemize}
We see that the number of functions in $X(\{\cB_j\}_{j=0}^3|_\gph)$ (39) is significantly less than  those in $X(\{\cB_j\}_{j=0}^3)$ (95). We remark that the 26 functions from $\cB_2$ could be further reduced in view of the non-effective blocks in $\cB_3$. For example, the block $B_6$ has 9 sub-blocks (\#16, \#17, \#18, \#22, \#23, \#24, \#28, \#29, \#30). These produce 8 functions  $\psi^{(\ell_1,\ell_2)}_{2,B_6}$ in $X(\{\cB_j\}_{j=0}^3|_\gph)$. But since the block $B_{29}$ representing the vertex $d$ is the only effective block,  one function, say $\psi_{2,B_6}^{(29,30)}$, is indeed enough. The total number of functions in $X(\{\cB_j\}_{j=0}^3|_\gph)$ could be reduced to $20=7(\mbox{in } \cB_0)+6(\mbox{in }\cB_1)+7(\mbox{in }\cB_2)$.


\section{Final remarks}
\label{sec:remarks}
We conclude the paper by some remarks and potential future work.
\begin{itemize}
\item[{\rm 1)}] In \cite{treepap,CMZ:ACHA:18}, the systems used to represent graph (digraph) signals are orthonormal. They are obtained through the tensor product approach of two 1D orthonormal systems coming from the Gram–Schmidt process. The tensor product orthonormal system is used for the digraph signal representation. In this paper, we do not use the Gram-Schmidt process but simply rely on the selection of  two sub-blocks. The resulted systems are tight frame systems and could have more directionality than the tensor product approach orthonormal systems. Note that when each block has at most two children sub-blocks, our resulted system is also orthonormal. However, such an orthonormal system is different to those in \cite{treepap,CMZ:ACHA:18}.

\item[{\rm 2)}] Our construction of the system $X(\{\cB_j\}_{j\in\N_0})$ is not necessarily restricted to compact subsets $K\subseteq\R^d$. It is possible to extend the construction of the system on any $\sigma$-finite  measurable subsets of $\R^d$. More generally, our construction could be extended to abstract measurable spaces as well as compact Riemannian manifolds.

\item[{\rm 3)}]Since it is Haar-type, the framelet $\psi\in \Psi_{j,B}$ has only vanishing moment of order 1: $\int_K\psi(x) dx=0$. To increase the sparse representation ability, one could consider the construction of similar tight framelet systems with higher vanishing moments.

\item[{\rm 4)}] The number of elements in $X(\{\cB_j\}_{j=0}^J|_\gph)$ could be significantly reduced  by discarding non-effective framelet functions as pointed out in the end of Example~2.

\item[{\rm 5)}] Since there is a natural nested structure of the block sequences, fast algorithms could be developed for  framelet transforms.

\end{itemize}

\section{Proofs}
\label{sec:proofs}
We provide proofs of Lemma~\ref{lem:A} and Lemma~\ref{lem:Psi_B} here.

\begin{proof}[Proof of Lemma~\ref{lem:A}]
We prove that $A^\top A=\Id_m$ by showing that the columns of $A$ are orthogonal.  Indeed,  the 2-norm of the $\ell$-th column of $A$ is
\[
\begin{aligned}
\sum_{i=0}^n |a_{i,\ell}|^2
&=|a_{0,\ell}|^2+\sum_{1\le i_1<i_2\le m} |a_{(i_1,i_2),\ell}|^2\\
&=|a_{0,\ell}|^2+\sum_{1\le i_1<\ell\le m} |a_{(i_1,\ell),\ell}|^2+\sum_{1\le \ell<i_2\le m} |a_{(\ell,i_2),\ell}|^2\\
&=b_\ell+\sum_{i_1< \ell }b_{i_1} + \sum_{i_2> \ell }b_{i_2}= 1,\quad 1\le \ell\le m.
\end{aligned}
\]
Moreover,  the dot product of the $\ell_1$-th column and the $\ell_2$-th column of $A$ with $\ell_1<\ell_2$ is
\[
\begin{aligned}
&\sum_{i=0}^n a_{i,\ell_1}a_{i,\ell_2}
=a_{0,\ell_1}a_{0,\ell_2}+\sum_{1\le i_1<i_2\le m} a_{(i_1,i_2),\ell_1}a_{(i_1,i_2),\ell_2}\\
&=\sqrt{b_{\ell_1}b_{\ell_2}}+\sum_{i_2>\ell_1}a_{(\ell_1,i_2),\ell_1}a_{(\ell_1,i_2),\ell_2}+
\sum_{i_1<\ell_1}a_{(i_1,\ell_1),\ell_1}a_{(i_1,\ell_1),\ell_2}\\
&\quad+
\sum_{i_2>\ell_2}a_{(\ell_2,i_2),\ell_1}a_{(\ell_2,i_2),\ell_2}
+
\sum_{i_1<\ell_2, i_1\neq \ell_1}a_{(i_1,\ell_2),\ell_1}a_{(i_1,\ell_2),\ell_2}\\
&=\sqrt{b_{\ell_1}b_{\ell_2}}+\sum_{i_2>\ell_1}a_{(\ell_1,i_2),\ell_1}a_{(\ell_1,i_2),\ell_2}=\sqrt{b_{\ell_1}b_{\ell_2}}+\sum_{i_2>\ell_1}\sqrt{b_{i_2}}(-\sqrt{b_{\ell_1}}\delta_{i_2,\ell_2})
\\
&=\sqrt{b_{\ell_1}b_{\ell_2}}-\sqrt{b_{\ell_2}}\sqrt{b_{\ell_1}}=0.
\end{aligned}
\]
Similarly, $\sum_{i=0}^n a_{i,\ell_1}a_{i,\ell_2}=0$ for $\ell_1>\ell_2$. Hence, $A^\top A = \Id_m$ is an identity matrix. $\blacksquare$
\end{proof}

\medskip

\begin{proof}[Proof of Lemma~\ref{lem:Psi_B}]
To prove that $\Psi_{B}$ is tight, by linearity,  it suffices to show that
\[
\psi^{(\ell_1,\ell_2)}=\sum_{1\le \ell_1'<\ell_2'\le m}\la \psi^{(\ell_1,\ell_2)}, \psi^{(\ell_1',\ell_2')}\ra\psi^{(\ell_1',\ell_2')}, \quad 1\le \ell_1<\ell_2\le m.
\]
In fact, from \eqref{def:psi_ell}, we have
$\psi^{(\ell_1,\ell_2)} = \sqrt{b_{\ell_2}}\gamma_{\ell_1}-\sqrt{b_{\ell_1}}\gamma_{\ell_2}
=\sum_{\ell=1}^{m}a_{(\ell_1,\ell_2),\ell}\gamma_\ell$,
where $(a_{(\ell_1,\ell_2),\ell})_{\ell=1}^{m}$ is the row  of the matrix $A$ in Lemma~\ref{lem:A}.  By $A^\top A= I$ and the orthogonality of $\{\gamma_\ell\setsep 1\le\ell\le m\}$,  we have
\[
\begin{aligned}
&\sum_{1\le \ell_1'<\ell_2'\le m}\la \psi^{(\ell_1,\ell_2)},\psi^{(\ell_1',\ell_2')}\ra\psi^{(\ell_1',\ell_2')}\\
=&\sum_{1\le \ell_1'<\ell_2'\le m}(\sum_{\ell}a_{(\ell_1,\ell_2),\ell}a_{(\ell_1',\ell_2'),\ell}\sum_{\tilde \ell}a_{(\ell_1',\ell_2'),\tilde\ell}\gamma_{\tilde\ell})\\
=&
\sum_{\ell}a_{(\ell_1,\ell_2),\ell}\sum_{\tilde \ell}\sum_{1\le \ell_1'<\ell_2'\le m}a_{(\ell_1',\ell_2'),\ell}a_{(\ell_1',\ell_2'),\tilde\ell}\gamma_{\tilde\ell}\\
=&\sum_{\ell}a_{(\ell_1,\ell_2),\ell}\sum_{\tilde\ell}(\delta_{\ell,\tilde\ell}-\sqrt{b_\ell b_{\tilde\ell}})\gamma_{\tilde\ell}
= \sum_{\ell}a_{(\ell_1,\ell_2),\ell}(\gamma_\ell-\sum_{\tilde\ell}\sqrt{b_\ell b_{\tilde\ell}}\gamma_{\tilde\ell})
\\
=& (\sqrt{b_{\ell_2}}\gamma_{\ell_1}-\sum_{\tilde\ell}\sqrt{b_{\ell_1}b_{\ell_2} b_{\tilde\ell}}\gamma_{\tilde\ell})-(\sqrt{b_{\ell_1}}\gamma_{\ell_2}-\sum_{\tilde\ell}\sqrt{b_{\ell_2}b_{\ell_1} b_{\tilde\ell}}\gamma_{\tilde\ell})\\
=& \sqrt{b_{\ell_2}}\gamma_{\ell_1}-\sqrt{b_{\ell_1}}\gamma_{\ell_2}=\psi^{(\ell_1,\ell_2)},
\end{aligned}
\]
where the $3^{\rm rd}$ equation follows from that
$\sqrt{b_\ell b_{\tilde\ell}}+\sum_{1\le \ell_1'<\ell_2'\le m}a_{(\ell_1',\ell_2'),\ell}a_{(\ell_1',\ell_2'),\tilde\ell}$
is the dot product of the $\ell$-th column and $\tilde\ell$-th column of $A$, which equals to $\delta_{\ell,\tilde\ell}$. Therefore, $\Psi_B$ is a tight frame for $\cW_B$. $\blacksquare$
\end{proof}

\bibliographystyle{spmpsci}      


\end{document}